\documentclass[a4paper, 12pt]{amsart}
\usepackage{amsfonts, amsthm, amssymb, amsmath}
\usepackage{mathrsfs,array}
\usepackage{xy}
\usepackage{hyperref}
\usepackage{verbatim}
\usepackage[latin1]{inputenc}
\input xy
\xyoption{all}

\newtheorem*{theoA}{Theorem A}
\newtheorem*{theoB}{Theorem B}
\newtheorem*{theoC}{Theorem C}

\newtheorem{theo}{Theorem}[section]
\newtheorem{prop}[theo]{Proposition}
\newtheorem{defi}[theo]{Definition}
\newtheorem{lemm}[theo]{Lemma}
\newtheorem{coro}[theo]{Corollary}
\newtheorem{conj}[theo]{Conjecture}

\theoremstyle{definition}
\newtheorem{rema}[theo]{Remark}

\newcommand{\mb}{\mathbb}
\newcommand{\mc}{\mathcal}
\newcommand{\mf}{\mathfrak}
\newcommand{\mbf}{\mathbf}
\newcommand{\mscr}{\mathscr}
\newcommand{\ra}{\rightarrow}
\newcommand{\mrm}{\mathrm}

\DeclareMathOperator{\Ad}{Ad}
\DeclareMathOperator{\Ext}{Ext}

\DeclareMathOperator{\Hom}{Hom}

\DeclareMathOperator{\tr}{tr}

\DeclareMathOperator{\Gal}{Gal}

\DeclareMathOperator{\Ind}{Ind}

\DeclareMathOperator{\GL}{GL}

\DeclareMathOperator{\Fil}{Fil}

\DeclareMathOperator{\nr}{nr}
\DeclareMathOperator{\cris}{cris}
\DeclareMathOperator{\nc}{nc}
\DeclareMathOperator{\an}{an}
\DeclareMathOperator{\ord}{ord}
\DeclareMathOperator{\rig}{rig}
\DeclareMathOperator{\cl}{cl}
\DeclareMathOperator{\wt}{wt}
\DeclareMathOperator{\sm}{sm}
\DeclareMathOperator{\alg}{alg}
\DeclareMathOperator{\Set}{Set}
\DeclareMathOperator{\univ}{univ}
\DeclareMathOperator{\Sen}{Sen}
\DeclareMathOperator{\per}{per}

\DeclareMathOperator{\lp}{lp}
\DeclareMathOperator{\Ver}{Ver}
\DeclareMathOperator{\pol}{pol}
\DeclareMathOperator{\la}{la}
\DeclareMathOperator{\id}{id}

\title[Ordinary representations for $\mrm{U}(3)$]{Ordinary representations and companion points for $\mrm{U}(3)$ in the indecomposable case}
\date {\today}
\author{John Bergdall and Przemys\l aw Chojecki}
\email{bergdall@math.bu.edu}
\email{chojecki@math.jussieu.fr}

\begin{document}
\begin{abstract}
We prove that certain $p$-adic Banach representations, associated to local ordinary Galois representations, constructed by Breuil and Herzig appears in the completed cohomology of a definite unitary group in three variables. This confirms part of their conjecture. Our main technique is making use of $p$-adic automorphic forms for definite unitary groups and the eigenvarieties which parameterize them.
\end{abstract}

\maketitle

\addtocontents{toc}{\protect\setcounter{tocdepth}{1}}
\tableofcontents

\section{Introduction}

The $p$-adic Langlands program for the group $\GL_2(\mbf{Q}_p)$ is now well understood. The local correspondence can be summarized by saying there is a bijection between local Galois representations $\rho_p: \Gal(\overline{\mbf{Q}}_p / \mbf{Q}_p) \ra \GL_2(\overline{\mbf Q}_p)$ and certain $p$-adic Banach spaces $\Pi(\rho)$ with an action of $\GL_2(\mbf{Q}_p)$. The reader should consult \cite{cdp} for the most recent results. The global portion of the problem is that the local correspondence is realized within the completed cohomology of a tower of modular curves \cite{em5}. Unfortunately, not much is known beyond $\GL_2(\mbf Q_p)$. For example, even if $E$ is a finite extension of $\mbf Q_p$ then the category of representations of $\GL_2(E)$ on $p$-adic Banach spaces remains mysterious.

Breuil and Herzig in \cite{bh} have recently pursued the construction of interesting $p$-adic Banach space representations associated to local Galois representations by concentrating on the ordinary case. Let $L/\mbf Q_p$ be a finite extension. For a representation $\rho_p: \Gal(\overline{\mbf Q}_p/\mbf Q_p) \rightarrow \GL_n(L)$ which is generic and ordinary Breuil and Herzig constructed an admissible continuous unitary representation $\Pi(\rho_p)^{\ord}$ on an $L$-Banach space by taking successive extensions of unitary principal series. Their recipe takes as key input the splitting behavior of $\rho_p$ and thus forsees compatibility between the Galois and automorphic sides of the $p$-adic Langlands program for $\GL_n(\mbf Q_p)$.

They have further conjectured that if $\rho_p$ gives rise to $\Pi(\rho_p)$ under a conjectural $p$-adic local Langlands correspondence for $\GL_n(\mbf Q_p)$ then $\Pi(\rho_p)^{\ord}$ should account for the maximal subrepresentation of $\Pi(\rho_p)$ which is built out of principal series alone. That there is, or may be, a discrepancy between $\Pi(\rho_p)$ and $\Pi(\rho_p)^{\ord}$ is an interesting feature of the situation beyond $\GL_2(\mbf Q_p)$. Nevertheless, this paper is concerned with the representation $\Pi(\rho_p)^{\ord}$ and the insights it brings to the $p$-adic local Langlands program.

\subsection*{Results}
In this work we explore global aspects of  \cite{bh}. Our main results are a generalization, from ${\GL_2}_{/\mbf Q}$ to unitary groups in three variables, of the work of Breuil and Emerton \cite{be}. That work  showed, prior to Emerton's larger work \cite{em1} on local-global compatibility, that a tower of modular curves can detect the splitting behavior of the $p$-adic Galois representation attached to a $p$-ordinary Hecke eigenform.

Let us now quickly setup some notation in order to state our results. Let $F/F^+$ be a CM extension of number fields in which $p$ is totally split and denote $G = \mrm U(3)$ a definite unitary group in three variables attached to $F/F^+$. Let us fix a compact open subgroup $K^p \subset G(\mbf{A} ^{p\infty} _F)$. Since $G$ is compact at the real places, the interesting cohomological degree is zero. We define the completed cohomology group of Emerton
\begin{equation*}
\widehat{H}^0(K^p)_L = \left( \varprojlim _n \varinjlim _{K_p} H^0(G(F^+) \backslash G(\mbf{A}^{\infty}_{F^+})/K_pK^p), \mbf{Z}_p/p^n \mbf{Z}_p)\right) \otimes _{\mbf{Z}_p} L
\end{equation*}
where $K_p$ runs over open compact subgroups of $G(\mbf{Q}_p)$. Here $L$ is a finite extension of $\mbf Q_p$ which we allow to change at will. The space $\widehat H^0(K^p)_L$ can be seen as a model for $p$-adic automorphic representations on $G$. It is an $L$-Banach space with continuous commuting actions of the group $\GL_3(\mbf Q_p)$ and a Hecke algebra $\mc H(K^p)^{\nr}$.

If $\pi$ is an automorphic representation on $G$ then it has an associated (in the usual sense) global Galois representation $\rho = \rho _{\pi}: \Gal(\bar{F}/F) \ra \GL_3(L)$ (extending $L$ if necessary). If $\pi$ has tame level $K^p$ then $\rho_\pi$ is unramified away from a finite set depending on $K^p$. We let $\lambda: \mc H(K^p)^{\nr} \rightarrow L$ be the Hecke eigenvalue associated to $\pi$, equivalently the Frobenius eigenvalues of $\rho$ at the unramified places. We refer to $\rho$ as automorphic if it arises in this way.

For each place $v \mid p$ in $F^+$, we write $v = \tilde v \tilde v^c$ as places of $F$, and consider the local Galois representation $\rho_v:=\rho_{\tilde v} :  \Gal(\overline{F}_{\tilde v}/{F_{\tilde v}}) \simeq \Gal(\overline{\mbf Q}_p/\mbf Q_p) \rightarrow \GL_3(L)$. If $\rho_{\tilde v}$ is generic and ordinary then the same is true for $\rho_{\tilde v^c}$. The $\GL_3(\mbf Q_p)$-representation $\Pi(\rho_{\tilde v})^{\ord}$ will be completely described in Section \ref{subsec:unit-comp} but for now we note that it depends only on $v \mid p$ in $F^+$ and that its socle filtration reflects the decomposability of $\rho_v$. We denote it simply by $\Pi(\rho_v)^{\ord}$. The following is our main result (see Theorem \ref{thm:we-did-it}). It is a weak form of the \cite[Conjecture 4.2.2]{bh}.

\begin{theoA}
Let $\rho$ be automorphic and suppose that for all $v \mid p$, $\rho_v$ is generic ordinary and totally indecomposable. Then there exists a closed embedding
\begin{equation*}
\widehat{\bigotimes}_{v \mid p} \Pi(\rho_v)^{\ord} \hookrightarrow \widehat H^0(K^p)_L ^{\lambda}.
\end{equation*}
\end{theoA} 
Here, the superscript $\lambda$ on the right-hand side refers to the $\lambda$-eigenspace for the Hecke action. Totally indecomposable means, by definition, that every subquotient of $\rho_v$ is indecomposable. Some authors may use the term ``maximally non-split''. We note that Breuil and Herzig were able to independently prove Theorem A (in a revision of their preprint, see  \cite[Theorem 4.4.8]{bh}) under  additional hypothesis, most notably that the the mod $p$ reduction $\overline \rho_v$ is totally indecomposable at each place $v \mid p$ as well. However, our technique is completely different from theirs and we have also been able to weaken the hypothesis to allow just indecomposability in the case where $F^+ = \mbf Q$. That is the following result (see Theorem \ref{theo:F^+=Q}). It is also part of the conjectures of Breuil and Herzig.

\begin{theoB}\label{theo:theoB}
Suppose that $F^+ = \mbf Q$, $\rho$ is automorphic and that $\rho_p$ is generic, ordinary and indecomposable. Then there exists a closed embedding
\begin{equation*}
\Pi(\rho_p)^{\ord} \hookrightarrow \widehat H^0(K^p)_L ^{\lambda}.
\end{equation*}
\end{theoB}
In order to contrast the two results, it is helpful to note that if $\rho_p$ is totally indecomposable then $\Pi(\rho_p)^{\ord}$ contains a unique irreducible principal series whereas if $\rho_p$ is indecomposable, but not totally indecomposable, then the $\GL_3(\mbf Q_p)$-socle of $\Pi(\rho_p)^{\ord}$ is the {\em direct sum of two principal series}. Thus, Theorem B implies that the completed cohomology is detecting the decomposability of $\rho_p$ via the appearance of an extra principal series representation, just as in the case of ${\GL_2}_{/\mbf Q}$.

Let us now expose our method for proving Theorems A and B. Our techniques are both similar and different from \cite{be}. The main overlap is in purely representation-theoretic results while the main departure is in our ability to apply said results. For simplicity we restrict to the case where $F^+ = \mbf Q$ for the rest of the introduction.

The keystone for our approach is Emerton's $p$-adic interpolation of the Jacquet functor. To shorten notation we suppress the tame level and coefficient field and  denote $\widehat H^0_{\an}=\widehat H^0(K^p)_{\an}$ the locally analytic vectors appearing in the completed cohomology, and $\widehat H^{0,\lambda}_{\an}$ the $\lambda$-eigenspace for the action of $\mc H(K^p)^{\nr}$. Since this is a locally analytic representation of $\GL_3(\mbf Q_p)$ we can apply the Jacquet functor \cite{em2,em3} and obtain a locally analytic representation  $J_B(\widehat H^{0,\lambda}_{\an})$ of the diagonal torus $T$ (depending on the choice of Borel $B \supset T$). Emerton's general results (e.g. the main result of \cite{em3}) create a bridge between principal series appearing inside $\widehat H^{0,\lambda}_{\an}$ and certain $T$-eigenvectors appearing inside $J_B(\widehat H^{0,\lambda}_{\an})$. Thus in order to control the representation $\widehat H^0(K^p)^{\lambda}_{L}$ one should study eigensystems appearing in $J_B(\widehat H^{0,\lambda}_{\an})$.

Classical and $p$-adic automorphic forms for $G$ produce $T$-eigensystems inside the Jacquet module. In more technical language, we use an eigenvariety to study $J_B(\widehat H^{0,\lambda}_{\an})$. Reconsider the automorphic Galois representation $\rho = \rho_\pi$ as above. If we assume that $\rho$ is crystalline at $p$ then the orderings of crystalline eigenvalues of $\rho_p$ give naturally rise to locally analytic characters $\chi$ called refinements (see Section \ref{subsec:galois-refine}). A classical calculation shows that $T$-eigenspace $J_B^\chi(\widehat H^{0,\lambda}_{\an})$ is non-zero. 

If one considers all pairs $(\chi,\lambda)$, with $\chi$ a continuous (thus locally analytic character) of $T$ and $\lambda$ a character of $\mc H(K^p)^{\nr}$, for which $J_B^\chi(\widehat H^{0,\lambda}_{\an}) \neq 0$ then one arrives at a rough definition of the eigenvariety $X_{K^p}$ associated to $G$ and the tame level $K^p$. It is a rigid analytic variety, of equidimension three in this case, and can be interpreted as parameterizing finite slope $p$-adic automorphic forms for $G$. We recall its relevant properties in the text. It contains a multitude of points, called classical points, which arise from automorphic representations for $G$ as in the previous paragraph. 

In the text, see Definition \ref{defi:bad-points-new}, we define what it means for a point $(\chi,\lambda)$ to be ``bad''. The analogous definition, and the terminology, in the ${\GL_2}_{/\mbf Q}$ case is in \cite{be}. Our motivation comes from the theory of Verma modules and the definition reflects the famous results of  Verma, and Bernstein, Gelfand and Gelfand on composition series of Verma modules. To give some idea of what bad points look like, a classical $p$-refined eigenform $f$ is bad if and only if there exists an overconvergent companion form $g$ for $f$, i.e. a finite slope overconvergent $p$-adic modular form $g$ such that $\theta^{k-1}g = f$, where $k$ is the weight of $f$ and $\theta = qd/dq$ on $q$-expansions. Our main representation-theoretic result is the following adjunction formula (see Theorem \ref{theo:analytic-maps}) involving locally analytic principal series and the Jacquet module.

\begin{theoC}
Suppose that $(\chi,\lambda) \in X_{K^p}(L)$ is not bad. Then
\begin{equation*}
\Hom_{\GL_3(\mbf Q_p)}(\Ind(\chi)^{\an}, \widehat H^0(K^p)^\lambda_{L,\an}) \simeq J_B^\chi(\widehat H^0(K^p)^\lambda_{L,\an}).
\end{equation*}
\end{theoC}
Here $\Ind(-)^{\an}$ is a locally analytic induction from a Borel. To orient the reader, Emerton's definition and construction of the Jacquet functor, plus \cite{em3}, provides a non-zero map from the left-hand side to the right-hand side (and not much more in general). Thus in order to prove Theorem C we have to show that away from the bad locus one has an isomorphism. In the case of ${\GL_2}_{/\mbf Q}$ this is \cite[Th\'eor\`eme 5.5.1]{be}.

Returning the situation of Theorems A and B, we calculate in Section \ref{sec:adjunction} that the representations $\Pi(\rho_p)^{\ord}$ can be built out of taking universal unitary completions of certain locally analytic principal series. Thus, the adjunction formula in Theorem C allows us to exhibit $\Pi(\rho_p)^{\ord}$ inside Hecke eigenspaces of the completed cohomology by exhibiting certain {\em good} points on the eigenvariety. Here the case of Theorem A departs slightly from the case of Theorem B. 

In the case of Theorem A, we show using Kisin's famous result on the $p$-adic interpolation of crystalline periods, that enough of the classical points with fixed unramified eigensystems $\lambda$ are not bad. At each such point we apply Theorem C and computations already mentioned shows that the corresponding principal series are enough to build $\Pi(\rho_p)^{\ord}$ inside the completed cohomology.

In the case of Theorem B, where there is a non-trivial splitting behavior of the Galois representation at $p$, the not bad classical points on the eigenvariety are not sufficient to produce $\Pi(\rho_p)^{\ord}$. The phenomena here is more complex. What happens, and this happens in the ${\GL_2}_{/\mbf Q}$ case as well, is that the splitting behavior of $\rho_p$ is reflected on the eigenvariety by the existence of certain non-classical ``companion points". Those points, analogous to companion forms of \cite{be}, may either be bad or not. Under the hypothesis of Theorem B, the situation is rigid enough so that we can show that there are sufficiently many good classical and non-classical companion points to explain the appearance of $\Pi(\rho_p)^{\ord}$. 

The symmetry between the splitting of the Galois representation at $p$ and the construction of non-classical points on the eigenvariety is explained in Section \ref{sec:com-bad}. Our technique follows that of \cite{ber2}, wherein the geometry of the eigenvariety explains simultaneously the existence of companion points and the splitting of the Galois representation.

To end this introduction, let us remark briefly on the wider context of this work. As indicated following Theorem C, the adjunction formula reduces our weak form of the Breuil-Herzig conjecture to the existence or not of certain Hecke eigensystems in spaces of $p$-adic automorphic forms. In the totally indecomposable case, this amounts to showing that certain Hecke eigensystems do not exist in spaces of $p$-adic automorphic forms. The converse, constructing ``overconvergent companion forms'', is a more serious matter.  For general conjectures on the existence of companion points, see \cite{br} (Breuil has also given a different adjunction formula, see Remark \ref{rema:breuil}, and studied the socle of the completed cohomology in the non-ordinary case).

Some of our methods are general enough to be applicable in other contexts. In many places the restriction to $n = 3$ is only used for brevity and concreteness at the moment. For example, the proof of Theorem C is completely conceptual and a similar result in greater generality should follow easily. We plan to have a sequel dealing with general definite unitary groups. We could also deal with some non-compact Shimura varieties. We only signal here that we can also prove similar results in the non-compact case for unitary groups $\mrm U(2,1)$. Details for that will appear elsewhere.

\subsection*{Organization}

In Section \ref{sec:eigenvarieties} we explain the theory of eigenvarieties. We mostly work with the constructions of Emerton \cite{em1}. At some point we also make a reference to work of Chenevier. Thus we recall two separate constructions of eigenvarieties.

Section \ref{sec:com-bad} contains two main points. The first is an expansion of work of the first author \cite{ber2} on the relationship between the splitting of $p$-adic Galois representations and the existence of companion points via the geometry of $p$-adic families. The rest of Section \ref{sec:com-bad} discusses the notion of companion points and bad points and their mutual relation. In the final part of Section \ref{sec:com-bad} we specialize to the generic ordinary locus on the eigenvariety. Beginning in this section we make use of points satisfying the condition $(\dagger)$ (see Definition \ref{defn:dagger}). These points are the largest locus on which we can apply sensibly our adjunction formula and so we attempted to maintain that level of generality throughout.

The most technical section of our work is Section \ref{sec:adjunction}. It deals entirely with proving the adjunction formula Theorem C. We have taken some care to give a conceptual proof of the adjunction formula. Thus we have also included some minor results that are surely well-known but whose proofs we had no reference for. We hope to prove a general version of the adjunction formula in the sequel to this paper.

Section \ref{sec:breuil-herzig-section} contains the proofs of Theorems A and B. Here we recall the conjecture of Breuil and Herzig and we explain our computations. In fact, we also give a supplementary realization of their representations in terms of unitary completions of locally analytic principal series which may be of independent interest.

\subsection*{Notations} We introduce notations which we will use constantly throughout the text. We fix an isomorphism $\overline{\mbf{Q}}_p \simeq \mbf{C}$. 

In general we will use $G$ for the unitary group and 
\begin{equation*}
G_{\Sigma_p} = G(F^+\otimes_{\mbf Q} \mbf Q_p) = \prod_{v \mid p} \GL_3(\mbf{Q}_p).
\end{equation*}
We will have the Borel $B_{\Sigma_p}$ of {\em upper} triangular matrices and its opposite $B^-_{\Sigma_p}$. The diagonal torus is $T_{\Sigma_p}$. The modulus character of $\GL_3(\mbf Q_p)$ is $\delta_{B(\mbf Q_p)}: T \rightarrow \mbf{Q}_p^\times$ given by $|\cdot|^2 \otimes 1 \otimes |\cdot|^{-2}$. We also have the modulus character $\delta_{B_{\Sigma_p}}$ of $B_{\Sigma_p}$. 

For an algebraic weight $k$ we denote by $\delta _k$ the corresponding highest weight character of $T$.
Let $S_3$ be the symmetric group. It comes with a natural order (the ``Bruhat order'') in which the length of an element $\ell(\sigma)$ is the length of the shortest word expression $\sigma$ as a product of the two simple tranpositions (12) and (23). Thus this is the same as the length in the Weyl group for $\GL_3$.

It will be important for us to use the infinitesimal actions of the Lie algebras of these groups. Thus we use $\mf g_{\Sigma_p}$, $\mf b_{\Sigma_p}$, $\mf t_{\Sigma_p}$, $\mf n_{\Sigma_p}^-$ (for the opposite unipotent to $\mf n_{\Sigma_p} = \mf b_{\Sigma_p}/\mf t_{\Sigma_p}$) etc. to denote the corresponding Lie algebras. We use $\mf t_{\Sigma_p}^{\ast} = \Hom(\mf t_{\Sigma_p}, \overline{\mbf Q}_p)$ to denote the dual space to $\mf t_{\Sigma_p}$ (i.e. the linear dual space of $\mf t_{\Sigma_p}\otimes_{\mbf Q_p} \overline{\mbf Q}_p$). If $\chi \in \mc T_{\Sigma_p}(\overline{\mbf Q}_p)$ then its differential $d\chi$ is an element of $\mf t_{\Sigma_p}^{\ast}$.  If we drop the subscript $\Sigma_p$ it is because we are talking about the group $\GL_3(\mbf Q_p)$. 

If $k$ is a weight (not necessarily algebraic) of ${\mf g}$ and $\sigma \in S_3$, we use $\sigma \cdot \delta$ to denote the ``dot action''
\begin{equation*}
\sigma \cdot k = \sigma(k + \rho_0) - \rho_0,
\end{equation*}
where $\rho_0 = (1,0,-1)$ is the half-sum of the positive roots. Here $\sigma$ acts on the right hand side by $\sigma((k_i)) = (k_{\sigma(i)})$ and thus $\sigma \cdot k$ is a {\em right} action.

Let $z: \mbf Q_p^\times \rightarrow \mbf Q_p^\times$ be the identity character. We denote the cyclotomic character by $\varepsilon$ (as a Galois representation) and we normalize the local class field theory so that $\varepsilon = z|z|$ (as a character of $\mbf Q_p^\times$) with the Hodge-Tate weight $-1$. 

\subsection*{Acknowledgements} We would like to thank Christophe Breuil, Jean-Francois Dat, David Hansen, Florian Herzig and Jim Humphreys for useful discussions and correspondence related to this work. In particular, the first author would like to thank Breuil for his patience regarding the results of Section \ref{sec:com-bad} not appearing until now. 

We would also like to thank MSRI and the organizers of the ``Hot topics: Perfectoid spaces and their applications" conference in February 2014 for hosting a wonderful event where some of this work was carried out. The second author would also like to thank Boston University for the hospitality in the week following the MSRI conference.

\section{Eigenvarieties}\label{sec:eigenvarieties}

The goal of this section is to recall the theory of definite eigenvarieties. We will recall two separate explicit constructions. The first, and most important for our purposes, is Emerton's construction via $p$-adically completed cohomology \cite{em1}. For references to the literature, we recall an earlier (and completely different) construction due to Chenevier \cite{che2}. The fact that these constructions give the same rigid analytic varieties is well-known \cite{loe}.

We begin first, however, with local preliminaries on two separate notions of refinements. In the case of Galois representations, we explicitly highlight the case of crystalline, ordinary representations. 

\subsection{Refinements}\label{subsec:galois-refine}

Throughout this section, we suppose that $\rho: G_{\mbf Q_p} \rightarrow \GL_3(L)$ is a crystalline representation whose crystalline eigenvalues are distinct and lie in $L^\times$. We further assume that $\rho$ is regular, i.e. the Hodge-Tate weights of $\rho$ are distinct and we write them $h_1 < h_2 < h_3$. 

\begin{defi}
A refinement $R$ of $\rho$ is an ordering $R = (\phi_1,\phi_2,\phi_3)$ of the crystalline eigenvalues appearing in $D_{\cris}(\rho)$.
\end{defi}
If $\{\phi_1,\dotsc,\phi_i\}$ is a list of distinct crystalline eigenvalues for $\rho$ then we denote by $\wt(\phi_1,\dotsc,\phi_i)$ the Hodge-Tate weight of the line $D_{\cris}(\wedge^i \rho)^{\varphi=\phi_1\dotsb \phi_i} \subset D_{\cris}(\wedge^i\rho)$. It must be a sum of distinct Hodge-Tate weights for $\rho$. Thus a refinement $R$ defines an ordering $(s_1,s_2,s_3)$ of the Hodge-Tate weights by declaring $\wt(\phi_1) = s_1$, $\wt(\phi_1,\phi_2) = s_1+s_2$ and $s_3$ is the unique weight not equal to $s_1$ or $s_2$.
\begin{defi}\label{defi:weight-type}
If $R$ is a refinement then its weight-type is the permutation $\tau \in S_3$ such that $s_i = h_{\tau(i)}$. We say $R$ is non-critical if $\tau = 1$. Otherwise, we say $R$ is critical of weight-type $\tau$.\end{defi}

For example if $\rho$ is a direct sum of three crystalline characters than there are six refinements, completely determined by their weight-types $\tau \in S_3$. Thus one is non-critical and the other five are critical.

We now specialize to the case that $\rho$ is upper-triangularizable (and still regular with distinct crystalline eigenvalues). Thus we assume that
\begin{equation}\label{eqn:matrix-rep}
\rho \sim \begin{pmatrix}
\psi_1 & \ast & \ast \\
0 & \psi_2 & \ast \\
0 & 0 & \psi_3 
\end{pmatrix}
\end{equation}
with each $\psi_i$ a crystalline character. Without loss of generality we may also assume that $\psi_i$ has Hodge-Tate weight $h_i$ and we refer to $\rho$ as ordinary. If we denote by $\phi_{\psi_i}$ the crystalline eigenvalue of $\psi_i$ then $v_p(\phi_{\psi_i}) = h_i$. Since $D_{\cris}(\rho)$ is weakly admissible we have that
\begin{align*}
\wt(\phi_{\psi_{i}}) &\leq v_p(\phi_{\psi_i}) = h_i & \text{(for $i=1,2,3$)}\\
\wt(\phi_{\psi_{i}},\phi_{\psi_j}) &\leq v_p(\phi_{\psi_i})  + v_p(\phi_{\psi_j}) = h_i + h_j & \text{(for each pair $(i,j)$)}
\end{align*}
The representation \eqref{eqn:matrix-rep} fixes one particular refinement $R_{\nc} := (\phi_{\psi_1},\phi_{\psi_2},\phi_{\psi_3})$ of $\rho$. Any other refinement $R$ must be of the form $R = R_{\nc}^{\sigma}:= (\phi_{\psi_{\sigma(i)}})_i$ for some $\sigma \in S_3$.

\begin{lemm}\label{lemm:indecomp-weights}
If $\rho$ is indecomposable then $\wt(\phi_{\psi_3}) < h_3$ and $\wt(\phi_{\psi_2},\phi_{\psi_3}) < h_2+h_3$.
\end{lemm}
\begin{proof}
Each inequality is proved identically, so we only prove the first inequality. Suppose that $\wt(\phi_{\psi_3}) = h_3$. Then, the line $D_{\cris}(\rho)^{\varphi = \phi_{\psi_3}} \subset D_{\cris}(\rho)$ is weakly admissible and is a complement to the weakly admissible subspace $D_{\cris}(\rho)^{\varphi = \phi_{\psi_1}} \oplus D_{\cris}(\rho)^{\varphi = \phi_{\psi_2}}$. Thus $D_{\cris}(\rho)$ is a sum of two weakly admissible subspaces. This implies that $\rho$ is the direct sum of two subrepresentations, contradicting that $\rho$ be indecomposable.
\end{proof}

Returning to the representation as in \eqref{eqn:matrix-rep}, we say that $\rho$ is totally indecomposable (sometimes also called ``maximally non-split") if no conjugate of $\rho$ takes values in a proper Zariski closed subgroup of the upper triangular matrices. That is, if in every conjugate of $\rho$ as in \eqref{eqn:matrix-rep}, none of the extensions vanish.

\begin{lemm}\label{lemm:simple-trans-non-critical}
If $\rho$ is as in \eqref{eqn:matrix-rep} then $R_{\nc}$ is always non-critical. Furthermore, if $\rho$ is indecomposable then we have:
\begin{enumerate}
\item If $\rho$ is totally indecomposable then $R^{(12)}_{\nc}$ and $R^{(23)}_{\nc}$ are non-critical.
\item If $\Hom(\psi_2,\rho) \neq 0$ then $R^{\sigma}_{\nc}$ is critical if and only if $\sigma = (12)$ or $\sigma=(123)$. The refinements $R^{(12)}_{\nc}$ and $R^{(123)}_{\nc}$ are critical of weight-type $(12)$.
\item If $\Hom(\psi_3,\rho/\psi_1) \neq 0$ then $R^\sigma_{\nc}$ is critical if and only if $\sigma = (23)$ or $\sigma=(132)$. The refinements $R^{(23)}_{\nc}$ and $R^{(132)}_{\nc}$ are critical of weight-type $(23)$.
\end{enumerate}
\end{lemm}
\begin{proof}
Since $\wt(\phi_{\psi_1}) \leq h_1$ and $h_1$ is the least Hodge-Tate weight we have $\wt(\phi_{\psi_1}) = h_1$. Similarly, $\wt(\phi_{\psi_1},\phi_{\psi_2}) = h_1+h_2$. Thus $R_{\nc}$ is non-critical. 

Note now that $\Hom(\psi_2,\rho) \neq 0$ if and only if $D_{\cris}(\rho)^{\varphi=\phi_{\psi_2}}$ is a weakly-admissible subspace, if and only if $\wt(\phi_{\psi_2}) = h_2$. Thus if $\rho$ is totally indecomposable then $R_{\nc}^{(12)}$ is non-critical. A similar argument shows that $R_{\nc}^{(23)}$ is non-critical as well if $\rho$ is totally indecomposable.

Suppose now that $\Hom(\psi_2,\rho) \neq 0$ but $\rho$ is indecomposable. We will prove (2). By the previous paragraph, if $\sigma(1) = 2$ then $R^\sigma_{\nc}$ is critical. Furthermore, each such $R_{\nc}^{\sigma}$ is critical of weight-type $(12)$, by Lemma \ref{lemm:indecomp-weights}  if $\sigma=(123)$.  On the other hand, since $\rho$ is indecomposable, the complement $D_{\cris}(\rho)^{\varphi=\phi_{\psi_1}}\oplus D_{\cris}(\rho)^{\varphi=\phi_{\psi_3}}$ to $D_{\cris}(\rho)^{\varphi=\phi_{\psi_2}}$ is not weakly admissible, and thus $\wt(\phi_{\psi_1},\phi_{\psi_3}) = h_1+h_2$. Finally it is easy to see that $\wt(\phi_{\psi_3}) = h_1$. Indeed, if not then we would see that $\Fil^{h_2}D_{\cris}(\rho)$ is a $\varphi$-stable with crystalline eigenvalues $\phi_{\psi_2}$ and $\phi_{\psi_3}$. Since this would contradict Lemma \ref{lemm:indecomp-weights}, we have our claim. It follows from these observations that $R^\sigma_{\nc}$ is non-critical if $\sigma \in \{(132), (23), (13)\}$. This completes the proof of (2). The proof of (3) is done the same way and we leave it to the reader.
\end{proof}

\begin{rema}
If $\rho$ is totally indecomposable then it is possible that a $R^\sigma_{\nc}$ is critical if $\sigma \notin \{1,(12),(23)\}$.
\end{rema}

Let us introduce the following definition.  
\begin{defi}
We say that $\rho$ is generic ordinary if $\rho$ is ordinary and  $p\phi_{\psi_i} \neq \phi_{\psi_{i+1}}$ for $i=1,2$.
\end{defi}
We suppose that $\rho$ is ordinary as in \eqref{eqn:matrix-rep}. Using local class field theory we write each character $\psi_i$ as a character of $\mbf Q_p^\times$
\begin{equation*}
\psi_i = z^{-h_i}\nr(\phi_{\psi_i}).
\end{equation*}
It follows from the definition and this expression that $\rho$ is generic ordinary if and only if $\psi_i\psi_j \notin \{1,\varepsilon^{\pm 1}\}$ for each $i\neq j$. Thus $\rho$ is generic ordinary if and only if $\rho$ is generic ordinary in the sense of \cite[Section 3.3]{bh}.

We now consider the automorphic side. The analog of a crystalline $G_{\mbf Q_p}$-representation is an unramified principal representations of $\GL_3(\mbf Q_p)$. Denote by $B(\mbf Q_p)$ the upper triangular Borel in $\GL_3(\mbf Q_p)$ and $T(\mbf Q_p)$ the diagonal torus inside $B(\mbf Q_p)$. If $\theta$ is a smooth character of $T(\mbf Q_p)$ then we can form the smooth non-normalized induction $\Ind_{B(\mbf Q_p)}^{\GL_3(\mbf Q_p)}(\theta)^{\sm}$.  Each representation $\Ind_{B(\mbf Q_p)}^{\GL_3(\mbf Q_p)}(\theta)^{\sm}$ contains a unique admissible irreducible unramified constituent. Given an irreducible admissible representation $\pi$ of $\GL_3(\mbf Q_p)$, if $\pi$ is unramified then $\pi$ is such a Jordan-Holder factor. The character $\theta$ is well-defined up to the action of $S_3$ given by
\begin{equation}\label{eqn:auto-regularity}
\theta^{(\sigma)} = \theta^{\sigma}(\delta_{B(\mbf Q_p)}^{-1/2})^{\sigma}\delta_{B(\mbf Q_p)}^{1/2}.
\end{equation}
Here the twists on the right hand side are the usual way of permuting coordinates, for example
\begin{equation*}
(\theta_1\otimes \theta_2\otimes\theta_3)^\sigma := \theta_{\sigma(1)} \otimes \theta_{\sigma(2)}\otimes \theta_{\sigma(3)}.
\end{equation*}

\begin{defi}
Let $\pi$ be an unramified smooth admissible representation of $\GL_3(\mbf Q_p)$. A refinement of $\pi$ is the choice of a smooth character $\theta$ such that $\pi \subset \Ind_{B(\mbf Q_p)}^{\GL_3(\mbf Q_p)}(\theta)^{\sm}$.
\end{defi}

In the terminology of \cite{bc} a refinement is the choice of $\theta$ such that $\pi$ appears as a Jordan-Holder factor of $\Ind_{B(\mbf Q_p)}^{\GL_3(\mbf Q_p)}(\theta)^{\sm}$ and our refinement is their accessible refinement. In this language, every $\sigma \in S_3$ defines a refinement $\theta^{(\sigma)}$ but only some $\sigma$ define an accessible refinement. To that point, however, we have an equivalence
\begin{equation}\label{eqn:unramified-criterion}
\Ind_{B(\mbf Q_p)}^{\GL_3(\mbf Q_p)}(\theta)^{\sm} \text{ is unramified} \iff p^{j-i}{\theta_i(p) \over \theta_j(p)} \neq p^{\pm 1} \text{ for $i\neq j$}.
\end{equation}
If that is the case (which it will be in our applications) then $\pi = \Ind_{B(\mbf Q_p)}^{\GL_3(\mbf Q_p)}(\theta)^{\sm}$, every refinement is accessible and thus $\pi \subset \Ind_{B(\mbf Q_p)}^{\GL_3(\mbf Q_p)}(\theta^{(\sigma)})^{\sm}$ for all $\sigma \in S_3$.

\subsection{Definite eigenvarieties}\label{subsec:def-eigenvarieties}
We fix a totally real field extension $F^+$ of the rational numbers $\mbf Q$ and $F/F^+$ a CM extension.  We assume that $p$ is totally split in $F$ and we let $\Sigma_p$ be the set of places $v \mid p$ in $F^+$. For each $v \in \Sigma_p$ we fix the choice of a place $\tilde v$ above $v$.

Let $G = \mathrm{U}(3,F/F^+)$ be a definite unitary group over $F^+$ in three variables attached to $F/F^+$ which is unramified at places of $F^+$ away from a finite set $S$. We assume that $S$ does not contain any places from $\Sigma _p$. If $w$ is a place of $F^+$ split in $F$ and not belonging to $S$ then each choice $\tilde w$ of place over $w$ defines an isomorphism
\begin{equation*}
G(F_w^+) \overset{\tilde w}{\simeq} \GL_3(F_{\tilde w}).
\end{equation*}
In particular, for each $v \in \Sigma_p$ we have a fixed isomorphism $G(F_v^+) \overset{\tilde v}{\simeq} \GL_3(\mbf Q_p)$. Denote now 
\begin{equation*}
G_{\Sigma_p} = G(F^+\otimes_{\mbf Q} \mbf Q_p) \simeq \prod_{v\in\Sigma_p} \GL_3(\mbf Q_p).
\end{equation*} 
Under these identifications we define $T_{\Sigma_p} = \prod_{v \in \Sigma_p} T(\mbf Q_p)$ to be the diagonal torus, $B_{\Sigma_p}$ the upper triangular Borel and $B^{-}_{\Sigma_p}$ the lower triangular Borel. We denote as well $N^0_{\Sigma_p}=N(\mc O_{F^+}\otimes_{\mbf Z} \mbf Z_p)$. Finally, we let 
\begin{equation}\label{eqn:monoid}
T^+ _{\Sigma _p} = \{ t\in T_{\Sigma_p} \mid tN^0 _{\Sigma _p} t^{-1} \subset N^0 _{\Sigma _p} \}.
\end{equation}

Fix a compact open subgroup $K^p \subset G(\mbf A_{F^+}^{p\infty})$. We factor $K^p$ into a product $K^p = \prod_{v \notin\Sigma_p} K^p_v$. Choose a finite set of places $\Sigma^p$ disjoint from $\Sigma_p$ of $F^+$ such that if $w \notin \Sigma := \Sigma^p\cup\Sigma_p$ then $K_w^p$ is hyperspecial maximal compact in $G(F^+_w)$. We assume that $S \subset \Sigma ^p$. We write the above factorization as 
\begin{equation*}
K^p = \prod_{w \notin \Sigma^p} K^p_w \times \prod_{w \in \Sigma^p} K_w^p =: K^{p\Sigma^p}K_{\Sigma^p}^p
\end{equation*}
and define the unramified Hecke algebra 
\begin{equation*}
\mc H(K^p)^{\nr} := \mc H(G(\mbf A_{F^+}^{p\Sigma^p})//K^{p\Sigma^p}).
\end{equation*}
The places $\tilde v$ above $v \in \Sigma_p$ define isomorphisms $G(F_v^+)\simeq \GL_n(F_{\tilde v}) = \GL_3(\mbf Q_p)$. 

We now define $\mc T_v := \Hom_{C^0}(T_v,\mbf G_m^{\rig})$ and $\mc T_{\Sigma_p} = \prod_{v\in \Sigma_p} \mc T_v = \Hom_{C^0}(T_{\Sigma_p},\mbf G_m^{\rig})$. Any element of $\mc T_{\Sigma_p}$ is necessarily locally analytic. We denote by $\mc T_{\Sigma_p}^{\sm}$ the locally constant characters and 
\begin{equation}\label{eqn:alg-wts}
\mc T^{\alg}_{\Sigma_p} = \{\left((z_{1,v},z_{2,v},z_{3,v}) \mapsto z_{1,v}^{s_{1,v}}z_{2,v}^{s_{2,v}}z_{3,v}^{s_{3,v}}\right)_v : s_{i,v} \in \mbf Z \text{ for $i=1,2,3$ and $v \in \Sigma_p$}\}
\end{equation}
the algebraic characters. The product map $\mc T_{\Sigma_p}^{\sm} \times \mc T_{\Sigma_p}^{\alg}\rightarrow \mc T_{\Sigma_p}$ surjects onto the locally algebraic characters.

Finally, if $k \in \mbf Z^3_{\Sigma_p}$ such that $k_{1,v} \geq k_{2,v} \geq k_{3,v}$ for all $v \in \Sigma_p$ then we define the dominant weight
\begin{equation*}
\delta_k = (z_{1,v}^{k_{1,v}} \otimes z_{2,v}^{k_{2,v}} \otimes z_{3,v}^{k_{3,v}})_{v \in \Sigma_p} \in \mc T_{\Sigma_p}^{\alg}
\end{equation*}
and we let $\mc T_{\Sigma_p}^{\alg,\geq 0} \subset \mc T_{\Sigma_p}^{\alg}$ be the submonoid of all such characters.

Suppose that $\delta_k \in \mc T_{\Sigma_p}$ is a dominant weight for $G_{\Sigma_p}$ and let $W_k$ be the irreducible algebraic  representation of $G_{\Sigma_p}$ with highest weight $k$. Since $G$ is definite, $G(F^+)$ is compact and the space $\mc A_k(G,K^p)$ of automorphic forms of weight $k$ and tame level $K^p$ decomposes as a $\mc H(K^p)^{\nr}$-module
\begin{equation}\label{eqn:auto-forms}
\mc A_k(G,K^p) \simeq \bigoplus_{\pi_\infty \simeq W_k} (\pi_f^{K^{p\Sigma_p}})^{m(\pi)}
\end{equation}
with $\pi$ running over irreducible automorphic representations for $G(\mbf A_{F^+})$ and $m(\pi)$ the multiplicity of $\pi$ appearing in $L^2(G(F^+)\backslash G(\mbf A_{F^+}))$. If $\pi$ is an irreducible automorphic representation for $G(\mbf A_{F^+})$ of tame level $K^p$ we denote by $\lambda_{\pi}: \mc H(K^p)^{\nr} \rightarrow \overline{\mbf{Q}}_p$ the canonical character. 

If $v \in \Sigma_p$ we denote by $\pi_v$ the local component of $\pi$, a smooth, admissible representation of $G(F_v^+) \simeq \GL_3(\mbf Q_p)$. Write 
\begin{equation*}
\pi_{\Sigma_p} = \bigotimes_{v \in \Sigma_p} \pi_v.
\end{equation*}
We say that $\pi$ is unramified at $p$ if $\pi_{\Sigma_p}$ is unramified, or equivalently, each $\pi_v$ is unramified. If $\pi_{\Sigma_p}$ is unramified then a refinement $\theta$ is a $\Sigma_p$-tuple $\theta = (\theta_v)_{v \in \Sigma_p}$ where  $\theta_v$ is a refinement of $\pi_v$. Equivalently, it is a locally constant character $\theta \in \mc T_{\Sigma_p}^{\sm}$ such that $\pi_{\Sigma_p} \subset \Ind_{B_{\Sigma_p}}^{G_{\Sigma_p}}( \theta)^{\sm}$.

Using our isomorphism $\overline{\mbf Q}_p \simeq \mbf C$ we identify a place $v \in \Sigma_p$ with an infinite place $v_\infty \in \Sigma_\infty$. Thus for each $v \in \Sigma_p$ we have the infinite component $\pi_{v_\infty}$ of $\pi$. It is an irreducible algebraic representation of the compact group $G(F_{\infty_v}^+) \simeq \mathrm{U}(3,\mbf R)$ and thus has an associated dominant weight $k_v = (k_{1,v} \geq k_{2,v} \geq k_{3,v}) \in \mbf Z^3$. Thus, for a given $\pi$ unramified at $p$, the choice of a refinement $\theta$ defines a locally algebraic character $\chi:= \theta \delta_{k} \in \mc T_{\Sigma_p}$.

We now describe the eigenvariety $X = X_{K^p}$ of tame level $K^p$. Before we begin, we note that our description uniquely defines the space $X$ by \cite[Proposition 7.2.8]{bc} and is independent of the two constructions given in Section \ref{subsec:explicit-construct}.

The eigenvariety is a reduced rigid analytic space, equidimensional of dimension 3$\left|\Sigma_p\right|$, equipped with the following auxiliary structures
\begin{enumerate}\label{aux-structure}
\item a finite map $\chi: X \rightarrow \mc T_{\Sigma_p}$,
\item a character $\lambda: \mc H(K^p)^{\nr} \rightarrow \Gamma(X,\mc O_X^{\rig})$, and
\item a Zariski dense set of points $X_{\cl} \subset X(\overline{\mbf Q}_p)$.
\end{enumerate}
These structures interact in the following way. If $x \in X(\overline{\mbf Q}_p)$ we denote by $\chi_x = \chi(x) \in \mc T_{\Sigma_p}(\overline{\mbf Q}_p)$ and by $\lambda_x$ the induced character
\begin{equation*}
\mc H(K^p)^{\nr} \overset{\lambda}{\longrightarrow} \Gamma(X,\mc O_X^{\rig}) \overset{\operatorname{eval}_x}{\longrightarrow} \overline{\mbf Q}_p.
\end{equation*}
Then the natural map
\begin{align*}
X(\overline{\mbf Q}_p) &\rightarrow \mc T_{\Sigma_p}(\overline{\mbf Q}_p) \times \Hom(\mc H(K^p)^{\nr}, \overline{\mbf Q}_p)\\
x &\mapsto (\chi_x,\lambda_x).
\end{align*}
is injective and defines a bijection between $X_{\cl}$ and pairs $(\chi,\lambda_\pi) = (\theta\delta_k,\lambda_\pi)$ attached to classical automorphic representations $\pi$ for $G(\mbf A_{F^+})$ which are unramified at $p$ and the choice of a refinement $\theta$ described above. We will almost always refer to points $x \in X(\overline{\mbf Q}_p)$ by giving their uniquely determined pair $(\chi_x,\lambda_x)$.

If $x \in X$ then $\chi_x$ can have a weight part split off. If $v \in \Sigma_p$ we denote $\mc W_v = \Hom(T(\mc O_{F_v^+}), \mbf G_m^{\rig})$ and $\mc W_{\Sigma_p} = \prod_{v \in \Sigma_p} \mc W_v$.  There is a natural projection $\mc T_{\Sigma_p} \rightarrow \mc W_{\Sigma_p}$ and we let the weight map $\kappa: X \rightarrow \mc W_{\Sigma_p}$ be the composition
\begin{equation*}
X \overset{\chi}{\longrightarrow} \mc T_{\Sigma_p} \overset{\operatorname{proj}}{\longrightarrow} \mc W_{\Sigma_p}.
\end{equation*}
We make the same definition for $\mc W_{\Sigma_p}^{\alg}$ as in \eqref{eqn:alg-wts}.  The projection map induces a natural isomorphism $\mc T_{\Sigma_p}^{\alg} \simeq \mc W_{\Sigma_p}^{\alg}$. If $z = (\chi,\lambda)$ is a classical point associated to an automorphic representation $\pi$ of weight $k = (k_v)$ then $\chi = \theta\delta_k$ with $\theta$ smooth, so that $\kappa(z) = \delta_k$. 

\subsection{Explicit constructions}\label{subsec:explicit-construct}

The succinct description of eigenvarieties we have given is sufficient for some purposes. But in order to observe the phenomena predicted by Breuil and Herzig \cite{bh} via $p$-adic families, we need to recall now an explicit construction due to Emerton. We will also quickly describe another approach and give their relationship. 

We let $L/\mbf Q_p$ be a finite extension with ring of integers $\mc O_L$ and uniformizer $\varpi_L$. We preserve the notations and choices of the previous section. The $p$-adically completed cohomology of tame level $K^p$ and with coefficients in $L$ is by definition
\begin{equation*}
 \widehat{H}^0(K^p)_L = \left( \varprojlim _n \varinjlim _{K_p} H^0(G(F^+) \backslash G(\mbf{A}^{\infty}_{F^+})/K_pK^p, \mbf{Z}_p/p^n \mbf{Z}_p)\right) \otimes _{\mbf{Z}_p} L
\end{equation*}
where $K_p$ runs over compact open subgroups of $G(\mbf{Q}_p)$. This is an $L$-Banach space equipped with a continuous representation of $G_{\Sigma_p}\times \mc H(K^p)^{\nr}$. If $\lambda: \mc H(K^p)^{\nr} \rightarrow \overline{\mbf Q}_p$ is a character we denote by $\widehat H^0(K_p)_L^{\lambda}$ the corresponding $G_{\Sigma_p}$-representation on the $\mc H(K^p)^{\nr}$-eigenspace of the character $\lambda$.

Within the space $\widehat H^0(K^p)_L$ we can take the locally analytic vectors $\widehat H^0(K^p)_{L,\an}$ and then this supplies us with a locally analytic $G_{\Sigma_p}$-representation to which we can apply Emerton's Jacquet functor \cite{em2,em3}. The output is a locally analytic representation $J_{B_{\Sigma_p}}(\widehat H^0(K^p)_{L,\an})$ of $T_{\Sigma_p}$. If $\chi \in \mc T_{\Sigma_p}$ then we will use the notation $J_{B_{\Sigma_p}}^\chi(\widehat H^0(K^p)_{L,\an})$ to denote the corresponding $\chi$-eigenspace.

The space of $N^0_{\Sigma_p}$-invariants in the completed cohomology has a Hecke action of the monoid $T^+_{\Sigma_p}$ (as in \eqref{eqn:monoid}) \cite[Section 3.4]{em2}. Explicitly, if $t \in T_{\Sigma_p}^+$ and $v\in \widehat H^0(K^p)^{N^0_{\Sigma_p}}$ then
\begin{equation}\label{eqn:hecke-operator-pit}
\pi_tv ={1 \over (N^0_{\Sigma_p} : tN^0_{\Sigma_p}t^{-1})}  \sum_{n \in N^0_{\Sigma_p}/tN^0_{\Sigma_p}t^{-1}} ntv.
\end{equation}
One checks that this is well-defined and preserves locally analytic vectors. This action relates to the Jacquet functor via a Hecke-equivariant isomorphism \cite[Proposition 3.4.9]{em2}
\begin{equation*}
J^\chi_{B_{\Sigma_p}}(\widehat H^0(K^p)_{L,\an}) \simeq \widehat H^0(K^p)_{L,\an}^{N^0_{\Sigma_p},T^+_{\Sigma_p}=\chi}.
\end{equation*}
Emerton then constructed an eigenvariety \cite{em1} $X_E$ whose points are described as 
\begin{equation*}
X_{E}(L) =\big\{ (\chi,\lambda) \in \mc T_{\Sigma_p}(L) \times \Hom(\mc H(K^p)^{\nr}, L) :  J_{B_{\Sigma_p}}^{\chi}(\widehat H^0(K^p)_{L,\an}^{\lambda}) \neq 0\big\}.
\end{equation*}

An earlier construction of the eigenvariety relies on the use of locally analytic principal series of Iwahori level. Choose $\chi \in \mc T_{\Sigma_p}$, which we extend by inflation to a character of $B_{\Sigma_p}$. Consider the Iwahori subgroup $I_{\Sigma_p} \subset G(\mc O_{F^+}\otimes_{\mbf Z_p} \mbf Z_p)$ whose elements are upper triangular modulo $p$. The intersection $I_{\Sigma_p} \cap B^-_{\Sigma_p} =: \overline{N}_{\Sigma_p}^0$ has the structure of a rigid analytic space and we define the following space of locally analytic functions
\begin{equation*}
\mc C_\chi = \left\{\text{$f: I_{\Sigma_p}B_{\Sigma_p} \rightarrow \mbf Q_p$ :\parbox{9cm}{\centering $ f|_{\overline{N}_{\Sigma _{p}}^0 }$ is locally analytic and $f(gb) = \chi(b)f(g)$ if $g \in I_{\Sigma_p}B_{\Sigma_p}$ and $b \in B_{\Sigma_p}$}}\right\}.
\end{equation*}
Notice that the monoid $T^+_{\Sigma_p}$ preserves $\overline{N}_{\Sigma_p}^0$ under conjugation $t^{-1}\overline{N}_{\Sigma_p}^{0}t \subset \overline{N}_{\Sigma_p}^{0}$ and thus the monoid $M_{\Sigma_p} = I_{\Sigma_p}T^+_{\Sigma p}I_{\Sigma_p} \subset G_{\Sigma_p}$ acts on $\mc C_{\chi}$ by the formula $(mf)(x) = f(m^{-1}x)$.

We define then the space of $p$-adic modular forms of tame level $K^p$ and weight $\chi$ to be
\begin{equation*}
F(\mc C_\chi, K^p) := \left\{ f: G(F^+) \backslash G(\mbf A_f)/K^p \rightarrow \mc C_\chi : \text{\parbox{7cm}{\centering $f(gx) = x^{-1}f(g)$ for all $g \in G(\mbf A_f)$ and $x \in I_{\Sigma_p}$}}\right\}.
\end{equation*}
This group has an action of the monoid $M_{\Sigma_p}$ and thus also the action of the Atkin-Lehner algebra $\mc A_{\Sigma_p}$, i.e. the commutative subalgebra of $\mc H(G_{\Sigma_p}\slash \slash I_{\Sigma_p})$ generated by $M_{\Sigma_p}$. Chenevier \cite{che2} then also defines an eigenvariety $X_C$ whose points are given by
\begin{equation*}
X_C(L) = \big\{(\chi,\lambda) \in \mc T_{\Sigma_p}(L) \times \Hom(\mc H(K^p)^{\nr},L) : F(C_\chi,K^p)^{\lambda,\mc A_p} \neq 0\big\}.
\end{equation*}

\begin{prop}\label{prop:comparision}
There is a canonical isomorphism of rigid varieties $X_C \simeq X_E$ commuting with the weight map $\chi$, the Hecke eigensystem $\lambda$ and preserving the classical points $X_{\cl}$. More specifically, for a pair $(\chi,\lambda)$ we have an isomorphism of vector spaces
\begin{equation*}
F(C_{\chi},K^p)^{\lambda,\mc A_p} \simeq J_{B_{\Sigma_p}}^{\chi}(\widehat H^0(K^p))^{\lambda}_{L,\an}.
\end{equation*}
\end{prop}
\begin{proof}
The formula is given by \cite[Proposition 3.10.3]{loe}, which shows that the closed points of $X_C$ and $X_E$ are the same. The general comparison theorem \cite[Proposition 7.2.8]{bc} implies that the rigid spaces are themselves the same, together with the usual extra structures.
\end{proof}
We now drop the subscript and refer to either $X$ as {\em the} eigenvariety. The main use for introducing Chenevier's eigenvariety is we have access to the BGG resolution attached for locally analytic principal series \cite{jones}. Let $\delta_k \in \mc T_{\Sigma_p}^{\alg,\geq 0}$ be a dominant weight. We consider the unique irreducible algebraic representation $W_k$ of $G_{\Sigma_p}$ whose highest weight is $k$ and we note that by definition
\begin{equation*}
F(W_k^\vee, K^p) \simeq \mc A_k(G, K^p).
\end{equation*}
Any highest weight vector in $W_k$ defines a natural $M_{\Sigma_p}$-equivariant inclusion $W_k^\vee \hookrightarrow \mc C_k$ \cite[Section 4.1]{che2}, which induces a $\mc H(K^p)^{\nr}\otimes \mc A_{\Sigma_p}$-equivariant inclusion on $p$-adic automorphic forms $j_k: F(W_k^\vee,K^p) \hookrightarrow F(\mc C_k, K^p)$. Recall that we use $\ell(w)$ to denote the Bruhat length of elements $w \in (S_3)_{\Sigma_p}$ and $w\cdot k$ to denote the ``dot'' action. If $F$ is an $\mc H(K^p)^{\nr}$-module then we use $F^{(\lambda)}$ to denote the {\em generalized} $\lambda$-eigenspace.

\begin{prop}\label{prop:jones-bgg}
Suppose that $k \in \mc T_{\Sigma_p}^{\alg}$ is a dominant weight and $\theta \in \mc T_{\Sigma_p}^{\sm}$ is locally constant. Then there is a long exact sequence
\begin{multline*}
0 \rightarrow F(W_{k}^\vee,K^p)^{(\lambda),\mc A_p=\theta} \overset{j_k}{\longrightarrow} F(\mc C_{\theta\delta_k}, K^p)^{(\lambda),\mc A_p} \rightarrow \prod_{\ell(w) = 1} F(\mc C_{\theta{\delta_{w\cdot k}}}, K^p)^{(\lambda),\mc A_p} \rightarrow \dotsb\\
\dotsb \rightarrow \prod_{\ell(w) = 3|\Sigma_p|} F(\mc C_{\theta\delta_{w\cdot k}}, K^p)^{(\lambda),\mc A_p}\rightarrow 0.
\end{multline*}
\end{prop}
\begin{proof}
Since $\theta$ is smooth, it extends naturally to a character on $I_{\Sigma_p}B_{\Sigma_p}$ by the formula $\theta(xb) = \theta(b)$ and one has an $I_{\Sigma_p}B_{\Sigma_p}$-equivariant isomorphism $\mc C_{\chi} \otimes \theta \rightarrow \mc C_{\theta\chi}$ for any $\chi \in \mc T_{\Sigma_p}$. In particular, for any $\chi \in \mc T_{\Sigma_p}$ we have a natural Hecke equivariant isomorphism $F(C_{\chi},K^p)^{\mc A_p=\theta} \simeq F(C_{\theta\chi},K^p)^{\mc A_p}$. The result now follows from taking $\chi = \delta_k$ and applying the \cite[Theorem 35]{jones} (compare with \cite[Lemma 4.3]{che}).
\end{proof}

As an application we relate the generalized eigenspaces to the local geometry of the weight map $\chi: X \rightarrow \mc T_{\Sigma_p}$. In the following theorem we make use of a hypothesis on the global Galois representation at a point on the eigenvariety, for which the reader can refer to Section \ref{subsec:refined-family}.

\begin{theo}[Chenevier]\label{thm:chenevier-classicality}
Suppose that $z = (\chi_z,\lambda_z) \in X_{\cl}$ is a classical point and assume that the global representation $\rho_z$ is irreducible. If we factor $\chi_z = \delta_k \theta$ with $\theta$ a smooth character of $T_{\Sigma_p}$ then $j_k: F(W_k^\vee,K^p)^{(\lambda),\mc A_p=\theta} \rightarrow F(\mc C_{\theta\delta_k},K^p)^{(\lambda),\mc A_p}$ is an isomorphism if and only if the weight map $\chi$ is \'etale at $z$.
\end{theo}
\begin{proof}
Since this is contained in the proof of \cite[Theorem 4.8]{che}, we just sketch the deduction from that work. Let $z$ be as in the statement of the theorem. We consider the functions
\begin{equation*}
\delta^{\an}(z_0) := \dim F(\mc C_{\chi_{z_0}},K^p)^{(\lambda_{z_0}),\mc A_p}
\end{equation*}
defined on some small affinoid open neighborhood $z \in \Omega \subset X$ and
\begin{equation*}
\delta^{\cl}(z_0) := \dim F(W_{\delta_{\kappa(z_0)}}^{\vee},K^p)^{(\lambda_{z_0}),\mc A_p=\theta_{z_0}}
\end{equation*}
defined on $\Omega_{\cl} := \Omega \cap X_{\cl}$. Then the proof of \cite[Theorem 4.8]{che} (as well as the lemmas of \cite[\S 4.4]{che}) shows that, under the hypothesis that $\rho_z$ is irreducible, one can choose $\Omega$ sufficiently small so that:
\begin{itemize}
\item $z$ is the only geometric point of $\chi^{-1}(\chi(z)) \cap \Omega$.
\item The function $\chi_0 \mapsto \sum_{x \in \Omega \cap \chi^{-1}(\chi_0)} \delta^{\an}(x)$ is constant on $\Omega$.
\item The map $z_0 \mapsto \delta^{\cl}(z_0)$ is constant on $\Omega_{\cl} := \Omega \cap X_{\cl}$.
\item If $z_0 \in \Omega_{\cl}$ with $z_0 \neq z$ then $\delta^{\an}(z_0) = \delta^{\cl}(z_0)$ and $\chi$ is \'etale at $z_0 \in \Omega_{\cl}$.
\end{itemize}
It follows that
\begin{equation*}
\delta^{\an}(z) = \sum_{\chi(z_0) = \chi_0} \delta^{\cl}(z_0) = \delta^{\cl}(z)\cdot \deg_z \chi|_{\Omega}
\end{equation*}
for any classical weight $\chi_0$ sufficiently close to $\chi$ in $\mc T_{\Sigma_p}$. Thus we see that $\chi$ is \'etale at $z$ if and only if $\delta^{\cl}(z) = \delta^{\an}(z)$, which holds if and only if $j_k$ is an isomorphism.
\end{proof}

\begin{rema}
The previous result, in more generality, is also given in \cite[Theorem 9.9]{br}.
\end{rema}

\subsection{The refined family of Galois representations}\label{subsec:refined-family}
 The eigenvariety carries a $p$-adic family of $p$-adic Galois representations. If $\tilde w$ is a place of $F$ and $\rho$ is a representation of the Galois group $G_F$ then we use the notation $\rho_{\tilde w}$ to denote the restriction $\rho_{\tilde w} := \rho|_{G_{F_{\tilde w}}}$. In the case that $v \in \Sigma_p$ we write $\rho_v = \rho_{\tilde v}$ for the choice $\tilde v \mid v$ which we made in Section \ref{subsec:def-eigenvarieties}.
 
 Suppose that $\pi$ is a classical automorphic representation on $G(\mbf A_{F^+})$ of tame level $K^p$ and weight $k = (k_v)_{v \in \Sigma_p}$. The work of Blasius and Rogawski (and many others in a more general situation) associates to $\pi$ a $p$-adic Galois representation
\begin{equation*}
\rho_{\pi} : G_{F,\Sigma} \rightarrow \GL_3(\overline{\mbf Q}_p).
\end{equation*}
It satisfies local-global compatibility at each finite place in the following sense:
\begin{enumerate}
\item If $\tilde w \notin \Sigma$ is a place of $F$ split over $F^+$ then the Frobenius semisimple Weil-Deligne representation associated to $\rho_{\tilde w}$ corresponds to the representation $\pi_w|\det|^{-1}$ of $\GL_3(F_{\tilde w}) \simeq  G(F_{w}^+)$ under the local Langlands correspondence for $\GL_3(F_{\tilde w})$.
\item If $v \in \Sigma_p$ then $\rho_{\pi,v}$ is de Rham with Hodge-Tate weights $h_{i,v} = -k_{i,v} + i - 1$.
\item If $\pi_{\Sigma_p}$ is unramified then $\rho_{\pi,v}$ is crystalline for each $v \in \Sigma_p$. Moreover, the set of crystalline eigenvalues of $\rho_{\pi,v}$ are given by $\{p^2\theta_{1,v}(p), p\theta_{2,v}(p), \theta_{3,v}(p)\}$ for any choice of refinement $\theta$ for $\pi_{\Sigma_p}$.
\end{enumerate}
By Chebotarev's density theorem, the first condition classifies $\rho_\pi$ up to the semi-simplification.

Consider again the eigenvariety $X = X_{K^p}$. The above associates to each classical point $z=(\theta\delta_k,\lambda) \in X_{\cl}$ the Galois representation $\rho_z := \rho_{\pi}$, depending only on $\lambda$. The choice of a refinement for $\pi_{\Sigma_p}$ has the following interpretation. Then, for all $v$, the third point above implies that we have a canonical choice
\begin{equation}\label{eqn:refine-on-family}
R_{z,v} := (\phi_{1,v},\phi_{2,v},\phi_{3,v}) = (p^2\theta_{1,v}(p), p\theta_{2,v}(p), \theta_{3,v}(p))
\end{equation}
of a refinement, in the sense of Section \ref{subsec:galois-refine}, of the local crystalline representation $\rho_{z,v}$.

\begin{defi}\label{defi:defi-critical}
The weight type of a point $z \in X_{\cl}$ is the $\Sigma_p$-tuple $\tau = (\tau_v)_{v \in \Sigma_p}$ where $\tau_v$ is the weight-type of the refinement $R_{z,v}$ given in \eqref{eqn:refine-on-family}. We say that $z$ is non-critical if it is of weight-type $(1_v)_v$ and critical of weight-type $\tau$ otherwise.
\end{defi}

For $z \in X_{\cl}$, let $t_z: G_{F,\Sigma} \rightarrow \overline{\mbf Q}_p$ be the function $g \mapsto \tr(\rho_z(g))$. Then, $t_z$ is a three-dimensional pseudocharacter of $G_{F,\Sigma}$ and by \cite[Proposition 7.5.4]{bc}, the map $z \mapsto t_z$ extends to a global pseudocharacter
\begin{equation}\label{eqn:pseudocharacter-on-family}
t: G_{F,\Sigma} \rightarrow \Gamma(X,\mc O_{X}^{\rig}).
\end{equation}
In particular, for each point $x \in X(\overline{\mbf Q}_p)$, specializing the pseudocharacter $t$ at $x$ we get a three-dimensional pseudocharacter $t_x: G_{F,\Sigma} \rightarrow \overline{\mbf Q}_p$. By \cite[Theorem 1(2)]{tay} there is a unique semi-simple representation $\rho_x: G_{F,\Sigma} \rightarrow \GL_3(\overline{\mbf Q}_p)$ such that $\tr(\rho_x) = t_x$. By the first property of the Galois representation over $X_{\cl}$, at every $x=(\chi_x,\lambda_x)\in X$, the function $\tr(\rho_x)$ at unramified places agrees with the value of $\lambda_x$ on $\mc H(K^p)^{\nr}$ in the usual way.

Let us consider the issues at the place $v\in \Sigma_p$. Denote by $\eta_{\Sigma_p}$ the composition 
\begin{equation}\label{eqn:hodge-tate-weight-map}
X\overset{\kappa}{\longrightarrow} \mc W_{\Sigma_p} \overset{\log}{\longrightarrow} \prod_{v \in \Sigma_p} \left(\mbf G_a^{\rig}\right)^3 \overset{s}{\longrightarrow} \prod_{v \in \Sigma_p} \left(\mbf G_a^{\rig}\right)^3,
 \end{equation}
 where $s$ is the affine change of coordinates
\begin{align*}
\left(\mbf G_a^{\rig}\right)^3 &\overset{s}{\longrightarrow} \left(\mbf G_a^{\rig}\right)^3\\
(u_1,u_2,u_3) &\mapsto (-u_1, -u_2 + 1, -u_3 + 2).
\end{align*}
If $x \in X(\overline{\mbf Q}_p)$ then we write $\eta_{\Sigma_p}(x) = (\eta_{i,v}(x))_v \in \prod_{v \in \Sigma_p} \left(\overline{\mbf Q}_p\right)^{3}$. In this notation, if $z \in X_{\cl}$ then $\eta_{i,v}(z) = h_{i,v}$ is the $i$th Hodge-Tate weight at the place $\tilde v$. Thus by \cite[Lemma 7.5.12]{bc} we have that for a general $x \in X(\overline{\mbf Q}_{p})$, $\eta_{\Sigma_p}(x)$ is the $\Sigma_p$-tuple of Hodge-Tate-Sen weights for $\rho_{x}$.

If $x \in X(L)$ and $v \in \Sigma_p$ we denote by $\chi_{x,v} = \chi_{x,v,1}\otimes \chi_{x,v,2}\otimes \chi_{x,v,3}$ the $v$th coordinate of $\chi_x \in \mc T_{\Sigma_p}(L)$. For $i=1,2,3$ and $v \in \Sigma_p$ we define analytic functions $F_{i,v} \in \Gamma(X,\mc O_X^{\rig})$ by
\begin{equation}\label{eqn:F_i-defn}
F_{i,v}(x) = p^{4-2i}\chi_{x,v,i}(p).
\end{equation}
If $z = (\theta\delta_k,\lambda) \in X_{\cl}$ is a classical point then
\begin{align*}
p^{\eta_{i,v}(z)}F_{i,v}(z) &= p^{h_{i,v}}p^{4-2i}\theta_{i,v}(p)p^{k_{i,v}}\\
&= p^{3-i}\theta_{i,v}(p).
\end{align*}
Thus for $z \in X_{\cl}$ the collection $R_z = (R_{z,v})$ of refinements is given by 
\begin{equation*}
R_{z,v} = (p^{\eta_{1,v}(z)}F_{1,v}(z),p^{\eta_{2,v}(z)}F_{2,v}(z), p^{\eta_{3,v}(z)}F_{3,v}(z)).
\end{equation*}
In particular, $X$ together with the Galois pseudocharacter $t$ of $G_{F,\Sigma}$ forms a refined family of Galois representations in the sense of \cite[Ch. 4]{bc}. This implies that we have the analytic continuation of crystalline periods:

\begin{prop}\label{prop:crystalline-interpolation}
For each $x \in X(\overline{\mbf Q}_p)$, each $v \in \Sigma_p$ and each $i$ we have
\begin{equation*}
\Fil^0 D_{\cris}(\wedge^i\rho_{x,v}(\kappa_{1,v}(x)^{-1}\dotsb \kappa_{i,v}(x)^{-1}z^{i-1}))^{\varphi = F_{1,v}(x)\dotsb F_{i,v}(x)} \neq 0.
\end{equation*}
In particular, if $\kappa(x)$ is algebraic (as in \eqref{eqn:alg-wts}) then
\begin{equation*}
\Fil^{\eta_{1,v}(x)+\dotsb +\eta_{i,v}(x)}D_{\cris}(\wedge^i \rho_{x,v})^{\varphi = p^{\eta_{1,v}(x)}F_{1,v}(x)\dotsb p^{\eta_{i,v}(x)}F_{i,v}(x)} \neq 0.
\end{equation*}
\end{prop}
\begin{proof}
The version we have stated here may be deduced from \cite{kis-omf}. For deformation-theoretic purposes in the next section, we need \cite{bc}. New proofs, with stronger conclusions, can be found in \cite{liu,kpx}.
\end{proof}

\section{On companion and bad points}\label{sec:com-bad} 

In this section we introduce the notion of companion points to classical points and what we call, following \cite{be}, bad points on the eigenvariety. Over the classical locus, the relationship between the two concepts is that a point is bad if and only if it has a companion point. We note that the concept of badness is important beyond the classical locus, as it is possible that the companion point to a classical point is still bad (this phenomena is not visible for ${\mrm{GL}_2}_{/\mbf Q}$, cf.  the proof of \cite[Th\'eor\`eme 5.7.2]{be}).

The first two subsections deal with companion points. In order to construct companion points we use the techniques of \cite{ber2}. Thus our first result of this section describes the classical points on the eigenvariety where the weight map ramifies. When combined Proposition \ref{prop:jones-bgg} and Theorem \ref{thm:chenevier-classicality} we are able to produce a simple criterion for the existence of companion points (see Theorem \ref{thm:comp-pt-theorem}). 

The third subsection contains the definition of bad points. Here we are inspired by the notion of ``strongly linked'' weights in Lie theory and the eventual proof of our adjunction formula (see Theorem \ref{theo:analytic-maps}). 

In the final subsection, we specialize to the generic ordinary locus. We show that to certain generic ordinary points (which are, by definition, classical) there is a canonical choice of companion point and that this companion point is, in turn, not bad.

\subsection{Ramification of the weight map}
In this section we use the refined family of Galois representations to study the infinitesimal properties of the weight map $\kappa$. Our approach follows \cite{ber2} in the sense that we use deformation rings of Galois representations. See also \cite[Theorem 9.7]{br} for a different proof of Theorem \ref{theo:consant-weights}.

Fix a local representation $\rho: G_{\mbf Q_p} \rightarrow \GL_3(L)$ with $L/\mbf Q_p$ a finite extension. We consider the category $\mf{AR}_L$ of local Artin $L$-algebras with residue field $L$. If $A \in \mf{AR}_L$ then we denote by $\mf m_A$ the maximal ideal of $A$. 

Define the functor $\mf X_{\rho}^{\univ}: \mf{AR}_L \rightarrow \underline\Set$ whose $A$-points are given by deformations $\rho_A : G_{\mbf Q_p} \rightarrow \GL_3(A)$ as in \cite{maz}. In the case that $\rho$ is indecomposable, the functor $\mf X_\rho^{\univ}$ is pro-representable in $\mf{AR}_L$, i.e. there is a complete local noetherian $L$-algebra $R_{\rho}^{\univ}$ with residue field $L$ such that if $A \in \mf{AR}_L$ then $\Hom(R^{\univ}_\rho,A) = \mf X_\rho^{\univ}(A)$. The homomorphisms on the left-hand side are local, continuous morphisms.

Assume now that $\rho$ is crystalline with distinct crystalline eigenvalues and regular with integer Hodge-Tate-Sen weights $h_1 < h_2 < h_3$ as in Section \ref{subsec:galois-refine}. If $\rho_A \in \mf X_{\rho}^{\univ}(A)$ then we denote by $P_{\Sen,\rho_A} \in A[T]$ the Sen polynomial of $\rho_A$. By Hensel's lemma and the assumption that $P_{\Sen,\rho}$ has distinct roots, the polynomial $P_{\Sen,\rho_A}$ splits into linear factors. We denote by $h_{i,\rho_A} \in A$ the unique root of $P_{\Sen,\rho_A}$ such that $h_{i,\rho_A} \equiv h_i \bmod \mf m_A$. Since $A$ is Artinian and the cyclotomic character takes values in $\mbf Q_p^\times$ we can naturally consider twists of $A$-valued representations by the $h_{i,\rho_A}$-th power of the cyclotomic character.

We now make the choice of a refinement $R$ of $\rho$. Recall that this means we have chosen an ordering $(\phi_1,\phi_2,\phi_3)$ for the crystalline eigenvalues appearing in the three-dimensional $\mbf Q_p$-vector space $D_{\cris}(\rho)$. We define $F_i := p^{-h_i}\phi_i$ for $i=1,2,3$ and notice two things. First, $h_i$ is the $i$th Hodge-Tate weight and does not depend on the refinement $R$, whereas $F_i$ does. Second,  $F_1\dotsb F_i$ is a simple eigenvalue for the crystalline Frobenius on $D_{\cris}((\wedge^i \rho)(h_1+\dotsb +h_i))$ for $i=1,2,3$.

With this in mind we define a subfunctor $\mf X_{\rho,R}^{\per} \subset \mf X_{\rho}^{\univ}$ as follows. A deformation $\rho_A$ belongs to $\mf X_{\rho,R}^{\per}(A)$ if and only if  for each $i$ we have $\Fil^0 D_{\cris}(\wedge^i \rho_{A}(h_{1,\rho_A}+\dotsb+h_{i,\rho_A}))^{\varphi = F_{1,A}\dotsb F_{i,A}}$ is free of rank one over $A$ for some tuple $(F_{1,A},F_{2,A},F_{3,A})$ of elements in $A$ such that $F_{j,A} \equiv F_j \bmod \mf m_A$. The functor $\mf X_{\rho,R}^{\per}$ is a relatively representable subfunctor of $\mf X_{\rho}^{\per}$ by \cite[Proposition 8.13]{kis-omf}. Thus, in the case that $\rho$ is indecomposable we have a universal deformation ring $R_{\rho,R}^{\per}$ parameterizing deformations whose successive exterior powers have interpolations of crystalline eigenvalues (cf. Proposition \ref{prop:crystalline-interpolation}).

This section so far has dealt only with representations of the local Galois group $G_{\mbf Q_p}$. We now turn our focus to representations of global Galois groups. For concreteness, we return to the eigenvariety $X$ as in the previous sections and consider a classical point $z \in X_{\cl}(L)$, with $L$ a finite extension of $\mbf Q_p$. The point $z$ has associated to it the global Galois representation $\rho_z: G_{F,\Sigma} \rightarrow \GL_3(L)$. We consider the functor $\mf X_{z}^{\per}: \mf{AR}_L \rightarrow \underline{\Set}$ whose $A$-points parameterize deformations $\rho_A: G_{F,\Sigma} \rightarrow \GL_3(L)$ such that for each place $v \in \Sigma_p$ the deformation $\rho_{A,v}$ of $\rho_{z,v}$ defines a point of $\mf X_{\rho_{z,v},R_{z,v}}^{\per}(A)$ where $R_{z,v}$ is the natural refinement \eqref{eqn:refine-on-family} at the point $z \in X(L)$.

\begin{prop}\label{eqn:surjective-def-ring}
Assume that $z \in X$ is a classical point such that 
\begin{enumerate}
\item $\rho_z$ is absolutely irreducible and
\item for each $v \in \Sigma_p$, $\rho_{z,v}$ has distinct crystalline eigenvalues.
\end{enumerate}Then the functor $\mf X_{z}^{\per}$ is representable by a complete local noetherian ring $R_{z}^{\per}$ and there is a natural surjective map $R_{z}^{\per} \twoheadrightarrow \widehat{\mc O}_{X,z}^{\rig}$.
\end{prop}
\begin{proof}
The functor $\mf X_{z}^{\per}$ being representable follows from the irreducibility hypothesis on $\rho_z$, each local deformation functor $\mf X_{\rho_{z,v},R_{z,v}}^{\per}$ being relatively representable over $\mf X_{\rho_{z,v}}$ and basic properties of fiber products of relatively representable functors.

Now consider the pseudocharacter \eqref{eqn:pseudocharacter-on-family} over the family $X$. We specialize it to a pseudocharacter $t_z^{\rig}: G_{F,\Sigma} \rightarrow \mc O_{X,z}^{\rig}$ . Since $\rho_z$ is absolutely irreducible and $\mc O_{X,z}^{\rig}$ is Henselian, \cite[Corollarie 5.2]{rou} and \cite[Lemma 4.3.7]{bc} imply that there is a unique semisimple representation $\widehat{\rho}_z^{\rig}: G_{F,\Sigma} \rightarrow \GL_3(\widehat{\mc O}_{X,z}^{\rig})$ whose trace is $t_z^{\rig}$. By the properties of the Galois representation in Section \ref{subsec:refined-family}, the representation $\widehat{\rho}_z^{\rig}$ deforms $\rho_z$. Proposition \ref{prop:crystalline-interpolation} implies that $\widehat{\rho}_{z}^{\rig}$ defines a point $R_z^{\per} \rightarrow \widehat{\mc O}_{X,z}^{\rig}$. The surjectivity of this map follows easily from the definition of $\widehat{\rho}_z^{\rig}$ and standard arguments.
\end{proof}

\begin{rema}
One may ask when is the surjection from Proposition \ref{eqn:surjective-def-ring} an isomorphism. After taking into account behavior at bad places and imposing a natural conjugate self-duality condition, a partial answer is  given in \cite{ber-thesis}.
\end{rema}

We will now use deformation rings to show that certain differences of weights are infinitesimally constant over $X$. Consider a point $z \in X_{\cl}$ and the weight-type $\tau_z \in (S_3)_{\Sigma_p}$ from Definition \ref{defi:defi-critical}. The point $z$ is non-critical, by definition, if and only if $\tau_{z,v} = 1$ for each $v \in \Sigma_p$. Denote by $T_{X,z}$ the Zariski tangent space of $X$ at the point $z$. For each place $v \in \Sigma_p$ we have the differential $d\eta_v = (d\eta_{i,v}): T_zX \rightarrow L^{\oplus 3}$.
\begin{theo}\label{theo:consant-weights}
If $z \in X_{\cl}$ such that
\begin{itemize}
\item $\rho_z$ is absolutely irreducible and
\item for each $v \in \Sigma_p$, $\rho_{z,v}$ has distinct crystalline eigenvalues
\end{itemize}
then $d(\eta_{i,v}-\eta_{\tau_{z,v}(i),v}) = 0$ for each $v \in \Sigma_p$.
\end{theo}
\begin{proof}
Fix a place $v \in \Sigma_p$ and consider the deformation rings $R_{z}^{\per}$ and $R_{\rho_{z,v},R_{z,v}}^{\per}$. We denote the Zariski tangent spaces by $T_z^{\per} = \mf X_{z}^{\per}(L[\varepsilon])$ and $T_{\rho_{z,v},R_{z,v}}^{\per} = \mf X_{\rho_{z,v},R_{z,v}}^{\per}(L[\varepsilon])$.  Each deformation $\rho_A \in \mf X_{\rho_{z,v}}^{\univ}(A)$ has Hodge-Tate-Sen weights $h_{i,\rho_A} \in A$ and we can study their infinitesimal deformations, which we also denote by $d\eta = (d\eta_{i}): T_{\rho_{z,v},R_{z,v}}^{\per} \rightarrow L^{\oplus 3}$. The notation is justified as \cite[Lemma 7.5.12]{bc} implies that we have a commutative diagram
\begin{equation*}
\xymatrix{
T_zX \ar[dr]_-{d\eta} \ar[r] & T_{z}^{\per} \ar[r] & T_{\rho_{z,v},R_{z,v}}^{\per} \ar[dl]^-{d\eta}\\
 & L^{\oplus 3}
}
\end{equation*}
Thus it is enough to show that $d(\eta_i - \eta_{\tau_{z,v}(i)}): T_{\rho_{z,v},R_{z,v}}^{\per} \rightarrow L^{\oplus 3}$ vanishes. We will handle the case of $i=1$ only. The case of $i=2$ follows from the same proof, either formally or by using the natural isomorphism $\wedge^2 \rho_{z,v} \simeq \rho_{z,v}^{\vee}(\det \rho_{z,v})$. The case of $i=3$ follows from the case of $i=1$ and $i=2$.

The rest of the proof will make use of the $(\varphi,\Gamma)$-module $D_{\rig}(\rho_{z,v})$ attached to $\rho_{z,v}$. A useful reference for this type of argument is \cite[Chapter 2]{bc}. Let $\mc R$ denote the Robba ring over $\mbf Q_p$ and $\mc R_L := \mc R \otimes_{\mbf Q_p} L$. We use $t$ to denote the usual element $t = \log(1+T)$ where $T$ is the variable on $\mc R$.

The refinement $R_{z,v} = (\phi_1,\phi_2,\phi_3)$ gives rise to a triangulation of $D_{\rig}(\rho_{z,v})$ and weight-type $\tau_{z,v}$. In particular there is an inclusion $\mc R_L(\delta_{1}) \subset D_{\rig}(\rho_{z,v})$ where $\delta_1: \mbf Q_p^\times \rightarrow L^\times$ is the character given by  $\delta_1 = z^{-h_{\tau_{z,v}(1)}}\nr(\phi_1)$. Its Hodge-Tate weight is $h_{\tau_{z,v}(1)}$. If $\tau_{z,v}(1) = 1$ then $\delta_1 = D_{\rig}(\psi_1)$, in the notation of \eqref{eqn:matrix-rep}, but otherwise $\delta_1$ is not \'etale.

The $L$-vector space $T_{\rho_{z,v},R_{z,v}}^{\per}$ is a subspace of (the canonically identified spaces) 
\begin{equation*}
\Ext^1_{G_{\mbf Q_p}}(\rho_{z,v},\rho_{z,v}) = T_{\rho_{z,v}}^{\univ} = \mf X_{\rho_{z,v}}^{\univ}(L[\varepsilon]).
\end{equation*}
Let $\widetilde \rho \in \mf X_{\rho_{z,v},R_{z,v}}^{\per}(L[\varepsilon])$. Denote by $D_{z,v}'$ the rank two $(\varphi,\Gamma)$-module 
\begin{equation*}
t^{-h_{1,v}}D_{\rig}(\rho_{z,v})/t^{-h_{1,v}}\mc R_L(\delta_1) \simeq D_{\rig}(\rho_{z,v}(h_{1,v}))/\mc R_L(z^{-h_{1,v}}\delta_1).
\end{equation*}
By functoriality we have a natural map 
\begin{equation*}
\alpha: \Ext^1_{G_{\mbf Q_p}}(\rho_{z,v},\rho_{z,v})\rightarrow \Ext^1_{(\varphi,\Gamma)}(\mc R_L(\delta_1),D_{\rig}(\rho_{z,v}))
\end{equation*}
 whose explicit definition is given by
\begin{equation}\label{eqn:alpha-defn}
\alpha(\widetilde{\rho}) = \ker(D_{\rig}(\widetilde{\rho}(h_{1,\widetilde \rho})) \twoheadrightarrow D_{\rig}(\rho_{z,v}(h_{1,\rho_{z,v}})) \twoheadrightarrow D_{z,v}').
\end{equation}
The image $\alpha(\tilde{\rho})$ is a rank four $(\varphi,\Gamma)$-module over $\mc R_L$. Notice that 
\begin{equation*}
h_{\tau_{v}(1),\widetilde\rho} - h_{1,\widetilde \rho} = h_{\tau_v(1),\rho_{z,v}} - h_{1,\rho_{z,v}} + d(\eta_{\tau_{z,v}(1)}-\eta_1)(\widetilde \rho)\varepsilon
\end{equation*}
is a Hodge-Tate-Sen weight of $\widetilde \rho(h_{1,\widetilde \rho})$. Further, \eqref{eqn:alpha-defn} implies that $\alpha(\widetilde \rho)$ has a Hodge-Tate weight $h_{\tau_v(1),\rho_{z,v}}-h_{1,\rho_{z,v}}$ of multiplicity two. Thus we have that $\alpha(\tilde \rho)$ is Hodge-Tate if and only if the Sen operator on $D_{\Sen}(\alpha(\tilde \rho))$ is semisimple, if and only if $d(\eta_{\tau_{z,v}(1)}-\eta_1)(\widetilde \rho) = 0$. We will show the stronger assertion that $\alpha(\widetilde \rho)$ is crystalline.

Note that by definition $\alpha(\tilde \rho)$ contains the  $(\varphi,\Gamma)$-submodule $D_{\rig}(\rho_{z,v}(h_{1,v}))$ and so, since $D_{\cris}(-)$ is left exact, $D_{\cris}(\alpha(\tilde \rho))$ contains a three-dimensional subspace $D_{\cris}(\rho_{z,v}(h_{1,v}))$. We can also consider the exact sequence
\begin{equation}\label{eqn:gen-space-seq}
0 \rightarrow D_{\cris}(\alpha(\widetilde{\rho}))^{(\varphi=p^{-h_{1,v}}\phi_{1})} \rightarrow D_{\cris}(\widetilde{\rho}(h_{1,\widetilde \rho}))^{(\varphi=p^{-h_{1,v}}\phi_{1})} \rightarrow D_{\cris}(D_{z,v}')^{(\varphi=p^{-h_{1,v}}\phi_{1})}
\end{equation}
 of $L$-vector spaces coming from \eqref{eqn:alpha-defn},
where $(-)^{(\varphi = p^{-h_{1,v}}\phi_{1})}$ is the generalized eigenspace (i.e. the subspace on which $\varphi - p^{-h_{1,v}}\phi_{1}$ acts nilpotently). Since the crystalline eigenvalues of $\rho_{z,v}$ are distinct, the far right space in \eqref{eqn:gen-space-seq} vanishes. Since $\tilde \rho \in \mf X_{\rho_{z,v},R_{z,v}}^{\per}(L[\varepsilon])$, the middle space is two-dimensional over $L$. Thus so is $ D_{\cris}(\alpha(\tilde \rho))^{(\varphi=p^{-h_{1,v}}\phi_{1})}$. However, this subspace can intersect $D_{\cris}(\rho_{z,v}) \subset D_{\cris}(\alpha(\widetilde \rho))$ in at most dimension one, again because $\varphi$ is semisimple on $D_{\cris}(\rho_{z,v})$. It follows that $\dim_L D_{\cris}(\alpha(\widetilde \rho)) \geq 4$ and thus $\alpha(\widetilde \rho)$ is crystalline as promised.
\end{proof}

\subsection{Companion points}\label{subsec:companion-pts}
In this section we define the notion of a companion point on the eigenvariety, generalizing \cite{ber2}. See \cite{br} for further information on overconvergent companion forms in higher dimensions.

\begin{lemm}\label{lemm:crys-alg-wts}
Fix a classical point $z = (\chi,\lambda) \in X_{\cl}$. Suppose that $x \in X(\overline{\mbf Q}_p)$ and $\lambda_x = \lambda_z$. Then $\rho_x \simeq \rho_z$, $\kappa(x) \in \mc W_{\Sigma_p}^{\alg}$ and there exists a $w = (w_v)_{v\in \Sigma _p} \in (S_3)_{\Sigma_p}$ such that $\eta(x)  = (\eta_{w_v(i),v}(z))_{i, v \in \Sigma_p}$.
\end{lemm}
\begin{proof}
The isomorphism $\rho_x \simeq \rho_z$ follows immediately from the equality $\lambda_x = \lambda_z$ and the Cebotarev density theorem. Once we know that $\kappa(x)$ is algebraic, we know the unordered list of integers appearing as its components as they are the unordered list of Hodge-Tate weights of $z$. The existence of $w$ is then clear.

The assertion that $\kappa(x)$ is algebraic follows from the analytic continuation of crystalline periods together with basic properties of the $(\varphi,\Gamma)$-modules $D_{\rig}(\rho_{x,v})$ associated to $\rho_x$ at each place $v \in \Sigma_p$. Specifically, Proposition \ref{prop:crystalline-interpolation} and \cite[Proposition 4.3]{col} imply\footnote{Strictly speaking, the reference \cite{col} is only in dimension two, but the particular form of the proposition that we are using holds in any dimension. One could also deduce this over an affinoid neighbourhood of $x$ using \cite[Section 6]{kpx}.} that for any $v \in \Sigma_p$ and any $i =1,2,3$, $D_{\rig}(\rho_{x,v})$ contains the character $\kappa_{i,v}(x)$, up to a crystalline twist. Since $\rho_{x,v}$ is crystalline for each $v$, the same is true for $\kappa_{i,v}(x)$, and this shows that $\kappa(x)$ is algebraic.
\end{proof}

\begin{rema}
It is worth pointing out in terms of the weight $\kappa$ what the permutation $w$ is doing. If $\kappa(z) = k$ then $\kappa(x)$ is obtained via the dot action $\kappa(x)= w\cdot k$. This follows from tracing through the definition of $\eta$ in terms of $\kappa$. Compare with Proposition \ref{prop:jones-bgg}.
\end{rema}

Recall that the natural projection map $\mc T_{\Sigma_p} \rightarrow \mc W_{\Sigma_p}^{\alg}$ induces an isomorphism between algebraic characters. If $\kappa \in \mc W_{\Sigma_p}$ is algebraic let us use $\kappa_{T_{\Sigma_p}}$ to emphasize we are viewing it in the group $\mc T_{\Sigma_p}^{\alg}$.

\begin{defi}\label{defi:comp-pts}
Suppose that $z = (\chi,\lambda) \in X_{\cl}$ is a classical point. We say that a point $x \in X(\overline{\mbf Q}_p)$ is a companion point for $z$ if
\begin{enumerate}
\item $x \neq z$,
\item $\lambda_x = \lambda_z$, and
\item $\chi_x\otimes \kappa(x)_{T_{\Sigma_p}}^{-1} = \chi_z\otimes \kappa(z)_{T_{\Sigma_p}}^{-1}$
\end{enumerate}
(Note that condition (2) and Lemma \ref{lemm:crys-alg-wts} imply that condition (3) is well-defined.)
\end{defi}
As we will see, the existence of companion points is related to geometric properties of the weight map $\chi$ at $z$, and also to the nature of the refinement associated to the point $z$. We begin with a short lemma.

\begin{lemm}\label{lemm:num-cons-comp}
Suppose that $z \in X_{\cl}$ is a classical point with refinement $R_{z,v} = (\phi_{1,v},\phi_{2,v},\phi_{3,v})$ and that $x$ is a companion point. Then $\rho_x \simeq \rho_z$ and $p^{\eta_{i,v}(x)}F_{i,v}(x) = p^{\eta_{i,v}(z)}F_{i,v}(z) = \phi_{i,v}$. Moreover, the permutation $w$ from Lemma \ref{lemm:crys-alg-wts} is non-trivial.
\end{lemm}
\begin{proof}
This follows immediately from Lemma \ref{lemm:crys-alg-wts} and the condition (3) of Definition \ref{defi:comp-pts}, along with the definitions \eqref{eqn:hodge-tate-weight-map} of $\eta$ and \eqref{eqn:F_i-defn} of the $F_i$. The non-triviality of $w$ follows since $x \neq z$ by definition.
\end{proof}

\begin{theo}\label{thm:comp-pt-theorem}
If $z \in X_{\cl}$ and $\rho_z$ is absolutely irreducible with distinct crystalline eigenvalues then the following are equivalent.
\begin{enumerate}
\item The point $z$ is critical.
\item The fiber of the weight map is non-trivial at $z$.
\item There exists a companion point for $z$.
\end{enumerate}
\end{theo}
\begin{proof}
The implication (1) implies (2) is given by Theorem \ref{theo:consant-weights}. The result that (2) implies (3) follows from combining Proposition \ref{prop:jones-bgg} and  Theorem \ref{thm:chenevier-classicality}. 

Finally we check (3) implies (1). So, suppose that $z$ has a companion point $x$. By Proposition \ref{prop:crystalline-interpolation} we have
\begin{equation*}
\Fil^{\eta_{1,v}(x)+\dotsb +\eta_{i,v}(x)}D_{\cris}(\wedge^i \rho_{x,v})^{\varphi = p^{\eta_{1,v}(x)}F_{1,v}(x)\dotsb p^{\eta_{i,v}(x)}F_{i,v}(x)} \neq 0.
\end{equation*}
Recall that the point $x$ and Lemma \ref{lemm:crys-alg-wts} defines the permutation $w = (w_v) \in (S_3)_{\Sigma_p}$ by $\eta_{i,v}(x) = \eta_{w(i),v}(z)$. By Lemma \ref{lemm:num-cons-comp}, $w$ is non-trivial and thus
\begin{equation*}
\Fil^{\eta_{\sigma_v(1),v}(z)+\dotsb +\eta_{\sigma_v(i),v}(z)}D_{\cris}(\wedge^i \rho_{z,v})^{\varphi = \phi_{1,v}\dotsb \phi_{i,v}} \neq 0.
\end{equation*}
Since $w_v$ is non-trivial for some $v$, the refinement $R_{z,v}$ is critical for some choice of $v$ and this shows (1).
\end{proof}

\subsection{Bad points}\label{subsec:bad-pts}

We now introduce bad points on the eigenvariety $X$. As the reader will note, the notion of badness could {\em a priori} be defined for points on eigenvarieties attached to a wide range of reductive groups $G$. To emphasize this, we have attempted to work as conceptually as possible while still maintaining the context of the rest of this paper.

Recall that we use German fraktur letters to denote Lie algebras corresponding to the groups we consider. For example, $\mf g_{\Sigma_p}$ it the Lie algebra associated to $G_{\Sigma_p}$. In particular, we will use $\mf t_{\Sigma_p}^{\ast}$ to denote the linear dual to the Lie algebra $\mf t_{\Sigma_p}$ (over a sufficiently large finite extension of $\mbf Q_p$ if needed). We use the notation $\alpha_1 = e_1 - e_2$ and $\alpha_2 = e_2 - e_3$ to denote the usual simple positive roots of $\mf{gl}_3$. The third positive root of $\mf{gl}_3$ is $\alpha_3 = \alpha_1 + \alpha_2$. Since $\mf g_{\Sigma_p} = (\mf{gl}_3)_{\Sigma_p}$, the roots of $\mf g_{\Sigma_p}$ are of the form $\alpha = (\alpha_v)_{v \in \Sigma_p}$ where $\alpha_v$ is a root of $\mf{gl}_3$ and $\alpha_v = 0$ except possibly at one place $v \in \Sigma_p$. A root $\alpha\neq 0$ is positive if and only if $\alpha_v > 0$ for the unique place $v \in \Sigma_p$ at which $\alpha_v \neq 0$.  Since the roots of $\mf g_{\Sigma_p}$ are all algebraic, we naturally confuse notation and speak about roots $\alpha \in \mf t_{\Sigma_p}^{\ast}$ and $\alpha \in \mc T_{\Sigma_p}$. If $\alpha$ is a root, denote by $\alpha^\vee$ the corresponding co-root. For each $\delta \in \mc T_{\Sigma_p}$, we then consider the locally analytic character $\langle \delta, \alpha^\vee \rangle:= \delta \circ \alpha^\vee: \mbf G_m^{\rig} \rightarrow \mbf G_m^{\rig}$. Let $\rho_0$ denote the half-sum of the positive roots.

\begin{defi}
A character $\delta \in \mc T_{\Sigma_p}$ is:
\begin{itemize}
\item $\alpha$-integral if $\langle \delta, \alpha^\vee\rangle$ is of the form $z \mapsto z^k$  for some $k \in \mbf Z$, in which case we use $\langle \delta, \alpha^\vee \rangle$ to also denote this integer, and
\item $\alpha$-dominant if $\delta$ is $\alpha$-integral and $\langle \delta + \rho_0, \alpha^\vee \rangle > 0$.
\end{itemize}
A character $\chi \in \mc T_{\Sigma_p}$ is locally $\alpha$-dominant if it is of the form $\chi = \theta \delta$ where $\theta$ is smooth and $\delta$ is $\alpha$-dominant.
\end{defi}

The Weyl group acts on weights $\mu \in \mf t_{\Sigma_p}^{\ast}$ by the formula
\begin{equation*}
w \cdot \mu = w(\mu + \rho_0) - \rho_0,
\end{equation*}
where the $w$ on the right hand side is the usual permutation action in the coordinates. This action preserves the $\mbf Z$-span of integral weights.

For each positive root $\alpha$ we have the associated reflection $s_\alpha \in (S_3)_{\Sigma_p}$ and the dot action of $s_\alpha$  on $\mf t_{\Sigma_p}^{\ast}$ extends to elements $\chi \in \mc T_{\Sigma_p}$ which are locally $\alpha$-integral in an apparent way. Indeed, when $\chi$ is locally $\alpha$-integral the weight $s_\alpha \cdot d\chi \in \mf t_{\Sigma_p}^{\ast}$ differs from $d\chi$ only by an algebraic weight. Thus there is a natural character $s_\alpha \cdot \chi \in \mc T_{\Sigma_p}$ such that $d(s_\alpha\cdot \chi) = s_\alpha \cdot \chi$. For example, if $\chi = \theta \delta$ with $\theta$ smooth and $\delta$ algebraic then $s_\alpha \cdot \chi = \theta \delta'$ with $\delta'$ algebraic as well. 

To give this example explicitly, if $\delta \in \mc T_{\Sigma_p}^{\alg}$ is of the form $\delta_{v} = z^{k_{1,v}}\otimes z^{k_{2,v}} \otimes z^{k_{3,v}}$ then
\begin{equation}\label{eqn:bad-pt-twist}
(s_\alpha \cdot \delta)_v = \begin{cases}
z^{k_1,v}\otimes z^{k_2,v} \otimes z^{k_3,v} & \text{if $\alpha_v = 0$}\\
z^{k_{2,v}-1} \otimes z^{k_{1,v}+1}\otimes z^{k_3,v} & \text{if $\alpha_v = \alpha_1$}\\
z^{k_{1,v}} \otimes z^{k_{3,v}-1}\otimes z^{k_{2,v}+1} & \text{if $\alpha_v = \alpha_2$}\\
z^{k_{3,v}-2} \otimes z^{k_{2,v}}\otimes z^{k_{1,v}+2} & \text{if $\alpha_v = \alpha_3$} 
\end{cases}
\end{equation}
This construction can be iterated to define a dot action of Weyl group elements $s_{\alpha_1}\dotsb s_{\alpha_r}$ on locally integral elements of $\mc T_{\Sigma_p}$. For the reader who is familiar with Verma modules (see \cite{hum}), the next definition will look familiar.

\begin{defi}\label{defi:strongly-linked}
Suppose that $\chi, \chi ' \in \mc T_{\Sigma_p}$. We write $\chi' \uparrow \chi$ if $\chi = \chi'$, or there exists a positive root $\alpha$ such that $\chi$ is locally $\alpha$-dominant and $s_\alpha \cdot \chi = \chi'$. We say that $\chi'$ is strongly linked to $\chi$ if there exists a sequence of positive roots ${\alpha_1},\cdots, {\alpha_r}$ such that
\begin{equation*}
\chi' = (s_{\alpha_1}\dotsb s_{\alpha_r})\cdot \chi \uparrow (s_{\alpha_2}\dotsb s_{\alpha_r})\cdot \chi \uparrow \dotsb \uparrow s_{\alpha_r}\cdot \chi \uparrow \chi.
\end{equation*}

\end{defi}
Note that being strongly linked is not an equivalence relation, but it is reflexive and transitive.

Let us explain Definition \ref{defi:strongly-linked} in the most interesting case where $\chi = \theta \delta_k$ and $\delta_k$ is a dominant weight. In that case, $s_\alpha \cdot \delta_k$ is defined for all positive roots $\alpha$. Moreover, if $k$ is regular then it is easy to see that $w\cdot \chi$ is defined for all elements $w \in S_3$ and that $w\cdot \chi$ is strongly linked to $\chi$ by a chain of length $\ell(w)$ equal to the length of $w$ in the Bruhat order on $S_3$.

We now return to the eigenvariety setting of the previous sections. Thus we have our eigenvariety $X_{K^p}$ of tame level $K^p$ parameterizing eigensystems $(\chi,\lambda) \in \mc T_{\Sigma_p} \times \mc H(K^p)^{\nr}$ acting on spaces of $p$-adic automorphic forms.

\begin{defi}\label{defi:bad-points-new}
Suppose that $z = (\chi,\lambda) \in X(\overline{\mbf Q}_p)$. We say that the point $z$ is bad if there exists a pair $(\chi',\lambda) \in X(\overline{\mbf Q}_p)$ such that $\chi' \neq \chi$ and $\chi'$ is strongly linked to $\chi$. If $z$ is bad, we denote by $w\cdot z := (w\cdot \chi, \lambda) \in X(\overline{\mbf Q}_p)$ (where $w\cdot \chi=\chi'$ is strongly linked to $\chi$) the corresponding companion point on $X(\overline{\mbf Q}_p)$.  
\end{defi}

The justification for the adoption of the term companion point from Section \ref{subsec:companion-pts} is given by the following result.

\begin{prop}\label{prop:bad-implies-critical}
Suppose that $z \in X_{\cl}$. Then $z$ is bad if and only if $z$ has a companion point. Specifically, if $w \cdot z \in X(\overline{\mbf Q}_p)$ for some $w \in (S_3)_{\Sigma _p}$ then $w \cdot z$ is a companion point in the sense of Definition \ref{defi:bad-points-new} if and only if it is a companion point in the sense of Definition \ref{defi:comp-pts}.
\end{prop}
\begin{proof}
It is clear from the definitions that if $w \cdot z \in X_{K^p}(\overline{\mbf Q}_p)$ then it is a companion point in the sense of Definition \ref{defi:comp-pts}. Similarly, if $z$ has a companion point as in Definition \ref{defi:comp-pts} then using Lemma \ref{lemm:crys-alg-wts} and Lemma \ref{lemm:crys-alg-wts} it must be of the form $w \cdot z$ for a unique $w \in (S_3)_{\Sigma_p}$, and $w \cdot z$ is a companion point as in Definition \ref{defi:bad-points-new} because $z$ has dominant weight (so that the strongly linked property holds).
\end{proof}

\begin{coro}
If $z \in X_{\cl}$ such that $\rho_z$ is absolutely irreducible and has distinct crystalline eigenvalues at each place above $p$ then $z$ is bad if and only if $z$ is critical.
\end{coro}
\begin{proof}
This follows from the previous proposition and Theorem \ref{thm:comp-pt-theorem}.
\end{proof}

\subsection{Generic ordinary points on eigenvarieties}\label{subsec:generic-ordinary}
We now illustrate on the previous sections in the case of generic ordinary points on $X$. Recall that a classical point $z = (\chi,\lambda) \in X_{\cl}$ is associated to the choice of a classical automorphic representation $\pi$ of $G(\mbf A_{F^+})$ with tame level $K^p$, spherical above $p$, and the choice of a smooth character $\theta$ such that
\begin{equation*}
\pi_{\Sigma_p} \subset \Ind_{B_{\Sigma_p}}^{G_{\Sigma_p}} (\theta)^{\sm}.
\end{equation*}
The character $\chi$ is then defined as  $\chi = \theta \delta_k$ where $k = (k_v)_{v \in \Sigma_p}$ is the list of dominant weights of $\pi_\infty$.  Fix for the moment an automorphic representation $\pi$ such that $\pi$ has tame level $K^p$, $\pi_{\Sigma_p}$ is unramified, $\pi$ has weight $k = (k_v)_{v \in \Sigma_p}$.
\begin{defi}
We say that $\pi$ is generic ordinary if $\rho_\pi$ is irreducible and $\rho_{\pi,v}$ is generic ordinary for all $v \in \Sigma_p$.
\end{defi}
One could give an intrinsic definition for admissible irreducible representations of $G_{\Sigma_p}$ but we prefer this vantage. Suppose that $\pi$ is generic ordinary and write
\begin{equation}\label{eqn:ordinary-matrix}
\rho_{\pi,v} \sim \begin{pmatrix} \psi_{1,v} & * & * \\ & \psi_{2,v} & * \\ & & \psi_{3,v}\end{pmatrix}
\end{equation}
with the characters $\psi_{i,v}$ crystalline of Hodge-Tate weights $h_{1,v} < h_{2,v} < h_{3,v}$ for each $v \in \Sigma_p$. If we denote $\phi_{\psi_{i,v}}$ the crystalline eigenvalue of $\psi_{i,v}$ then 
\begin{equation}\label{eqn:the-psis}
\psi_{i,v} = z^{-h_{i,v}}\nr(\phi_{\psi_{i,v}}),
\end{equation}
seen as a character of $\mbf Q_p^\times$ using local class field theory. Now consider the smooth character of $\mbf Q_p^{\times}$ given by
\begin{equation*}
\theta_v^{\nc} = |\cdot|^2\nr(\phi_{\psi_{1,v}}) \otimes |\cdot|\nr(\phi_{\psi_{2,v}}) \otimes \nr(\phi_{\psi_{3,v}}).
\end{equation*}
The character $\theta^{\nc} = (\theta_v^{\nc})_{v\in \Sigma _p}$ is a refinement of $\pi_{\Sigma_p}$ and defines a point $z_{\nc} = (\theta^{\nc}\delta_k, \lambda)\in X_{\cl}$. 

Since $\rho_{\pi,v}$ is assumed generic ordinary, it is easy to see $\theta_v^{\nc}$ satisfies the condition \eqref{eqn:unramified-criterion} for all $v \in \Sigma_p$. Thus the smooth irreducible representation $\Ind_{B(F_v^+)}^{G(F_v^+)} (\theta_v^{\nc})^{\sm}$ is unramified and we can explicitly list the other refinements of $\pi_{\Sigma_p}$ as follows. If $\sigma = (\sigma_v)_v \in (S_3)_{\Sigma_p}$ then we denote by $\theta^{\nc,(\sigma)} = (\theta_v^{\nc,(\sigma_v)})_{v \in \Sigma _p}$ the $\Sigma_p$-tuple of characters given by \eqref{eqn:auto-regularity}. By the criterion \eqref{eqn:unramified-criterion} we have that $\theta^{\nc,(\sigma)}$ defines a refinement of $\pi_{\Sigma_p}$ and all the refinements are of this form. Thus each of the points $z^{(\sigma)}_{\nc} = (\theta^{\nc,(\sigma)}\delta_k,\lambda) \in X_{\cl}$ satisfy $\rho_{z^{(\sigma)}_{\nc}}= \rho_\pi$.

\begin{defi}
A point $z \in X_{\cl}$ is called generic ordinary if there exists a generic ordinary $\pi$ such that $z = z_{\nc}^{(\sigma)}$ for some $\sigma \in (S_3)_{\Sigma_p}$ as above.
\end{defi}

We could equivalently phrase the definition as follows. Suppose that $z \in X_{\cl}$ is a point such that $\rho_z$ is irreducible and $\rho_{z,v}$ is generic ordinary for all $v \in \Sigma_p$. We then write $\rho_{z,v}$ as in \eqref{eqn:ordinary-matrix} and this necessarily defines a point $z_{\nc}=(\theta^{\nc}\delta_k,\lambda) \in X_{\cl}$; the point $z$ is then one of the points $z_{\nc}^{(\sigma)} = (\theta^{\nc,(\sigma)}\delta_k,\lambda)$. The tuple $\sigma$ is well-defined.

It will be convenient for us to realize, in terms of the Galois characters $\psi_{i,v}$, we have
\begin{equation*}
\theta_v^{\nc, (\sigma_v)} = (\psi_{\sigma_v(1)}z^{h_{\sigma_v(1)}} \otimes \psi_{\sigma_v(2)}z^{h_{\sigma_v(2)}} \otimes \psi_{\sigma_v(3)}z^{h_{\sigma_v(3)}})(|\cdot|^2 \otimes |\cdot| \otimes 1).
\end{equation*}
Evaluating at $p$ we get that
\begin{equation*}
p^{2-i}\theta_{i,v}^{\nc, (\sigma_v)}(p) = \phi_{\psi_{\sigma_v(i),v}}.
\end{equation*}
Thus in the generic ordinary case, to give a refinement of $\pi_{\Sigma_p}$ is the same as to give a collection $R = (R_v)_{v\in \Sigma_p}$ of refinements of each $\rho_{\pi,v}$.

For generic ordinary points on $X_{\cl}$ we can sometimes give a complete classification of their companion points. We note the following convention. Each simple root $\alpha$ of $\mf g_{\Sigma_p}$ gives rise to a reflection $s_\alpha$ in the Weyl group. When viewed as an element of $(S_3)_{\Sigma_p}$, $s_\alpha$ is of the form $s_\alpha = \tau_{v_0}$ for a transposition $\tau \in S_3$ and $v_0 \in \Sigma_p$. Conversely, such transpositions $\tau$ give rise to a positive root $\alpha_\tau$ so that $\tau = s_{\alpha_\tau}$ when viewed inside $(S_3)_{\Sigma_p}$. Our convention is to prefer to use $s_{\alpha_\tau}$ rather than $\tau$ when consider the dot action of Weyl elements on points of the eigenvariety as in Definition \ref{defi:bad-points-new}, whereas we may use $\tau$ itself to act on weights.

\begin{prop}\label{prop:indec-produce-comp-pt}
Suppose that $z = (\theta \delta_k, \lambda) \in X_{\cl}$ is generic ordinary and indecomposable above $p$, but critical of type $\tau$ where $\tau_{v_0} \neq (1)_{v_0}$ for a \underline{unique} place $v_0$. Then the pair $(\theta \delta_{\tau\cdot k},\lambda) \in \mc T_{\Sigma_p}(\overline{\mbf Q}_p) \times \Hom(\mc H(K^p)^{\nr},\overline{\mbf Q}_p)$ defines a point $s_{\alpha_\tau} \cdot z$ on $X(\overline{\mbf Q}_p)$ and $s_{\alpha_\tau} \cdot z$ is the unique companion point for $z$. Moreover, $\alpha_\tau$ is a simple root.
\end{prop}
\begin{proof}
Since $z$ is critical, Theorem \ref{thm:comp-pt-theorem} implies that a companion point $x$ exists. We must show that {\em any} companion point $x$ of $z$ has weight $\tau \cdot k$. This is implicit in the proof of Theorem \ref{thm:comp-pt-theorem} but let us make it explicit. 

By Lemma \ref{lemm:crys-alg-wts} we know that the point $x$ has algebraic weight $\kappa(x) = w\cdot k$ for some non-trivial $w \in (S_3)_{\Sigma_p}$. By Proposition \ref{prop:crystalline-interpolation} we have that for each place $v \in \Sigma_p$,
\begin{equation}\label{eqn:special-case-produce-comp-pt}
\Fil^{\eta_{w_v(1),v}(z)+\dotsb +\eta_{w_v(i),v}(z)}D_{\cris}(\wedge^i \rho_{z,v})^{\varphi = \phi_{1,v}\dotsb \phi_{i,v}} \neq 0.
\end{equation}
If $v \neq v_0$ then the weight-type $\tau$ satisfies $\tau_v = (1)$, and thus \eqref{eqn:special-case-produce-comp-pt} implies that $w_v(i) = i$ for each $v \neq v_0$. At the place $v_0$, the indecomposability hypothesis on $\rho_{z,v_0}$ and Lemma \ref{lemm:simple-trans-non-critical} imply that $\tau_{v_0}$ is a transposition. It is now easy to deduce that $\tau_{v_0} = w_{v_0}$, since $w_{v_0}$ is necessarily non-trivial.
\end{proof}

We now show that Proposition \ref{prop:indec-produce-comp-pt} together with Lemma \ref{lemm:simple-trans-non-critical} produces a large supply of generic ordinary points with a complete description of their companion points. We will use such points to study the conjecture of Breuil and Herzig (\cite{bh}) in Section \ref{subsec:breuil-herzig}.

\begin{defi}\label{defi:simple}
An element $\sigma \in S_3$ is called simple if $\sigma$ is one of $(1), (12)$ or $(23)$. If $\Sigma_p' \subset \Sigma_p$ is any subset then $\sigma \in (S_3)_{\Sigma_p}$ is $\Sigma_p'$-simple if $\sigma_v$ is simple for all $v \in \Sigma_p'$. A generic ordinary point $z \in X_{\cl}$ is called $\Sigma_p'$-simple if $z = z_{\nc}^{(\sigma)}$ with $\sigma$ a $\Sigma_p'$-simple element of $(S_3)_{\Sigma_p}$.
\end{defi}

\begin{coro}\label{prop:non-critcal-gen-ordinary}
Suppose that $z$ is generic ordinary, write $z = z_{\nc}^{(\sigma)}$ and let 
\begin{equation*}
\Sigma_p' = \{v \in \Sigma_p : \rho_{z,v} \text{ is totally indecomposable}\}.
\end{equation*}
If $z$ is $\Sigma_p'$-simple then the following statements are true:
\begin{itemize}
\item If $\Sigma_p' = \Sigma_p$ then $z$ is non-critical.
\item If $\Sigma_p - \Sigma_p' = \{v_0\}$ and $\rho_{z,v_0}$ is indecomposable then either $z$ is non-critical or possesses a unique companion point $s_{\alpha_\tau} \cdot z$ where $\tau$ is the weight-type of $z$.
\end{itemize}
\end{coro}
\begin{proof}
The first bullet point follows from Lemma \ref{lemm:simple-trans-non-critical} and the assumption that $\rho_{z,v}$ is totally indecomposable for all $v \in \Sigma_p$, along with the assumption that $z$ is $\Sigma_p$-simple. The second bullet point is the situation of Proposition \ref{prop:indec-produce-comp-pt}.
\end{proof}

It will be useful to adapt the previous corollary's conditions into a definition.

\begin{defi}\label{defn:dagger}
Let $z \in X_{\cl}$ and define $\Sigma_p' = \{v \in \Sigma_p : \rho_{z,v} \text{ is totally indecomposable}\}$. We say that a point $z$ satisfies $(\dagger)$ if
\begin{itemize}
\item $z$ is generic ordinary and $\rho_{z,v}$ is indecomposable for all $v \in \Sigma_p$.
\item $\left|\Sigma_p - \Sigma_p'\right| \leq 1$, i.e. $\rho_{z,v}$ is totally indecomposable for all but at most one place.
\item $z$ is $\Sigma_p'$-simple.
\item If $z$ is non-critical then $z$ is $\Sigma_p$-simple.
\end{itemize}
\end{defi}

As an illustration of the definition we will prove the following elementary lemma (which will be used in the proof of Proposition \ref{prop:relations-each-place}).

\begin{lemm}\label{lemm:closed-under-bruhat}
Suppose that $z_{\nc}^{(\sigma)}$ satisfies $(\dagger)$ and let $\tau$ be its weight-type. Then, if $\tau \leq \sigma' < \sigma$ with $\ell(\sigma') + 1 = \ell(\sigma)$ then $\sigma'_v = \sigma_v$ for all but one place $v_1$ at which $\sigma'_{v_1} = \tau_{v_1}$. Moreover, the point $z_{\nc}^{(\sigma')}$ satisfies $(\dagger)$ and has weight-type $\tau$ as well.
\end{lemm}
\begin{proof}
Since $\ell(\sigma') + 1 = \ell(\sigma)$ we know immediately that $\sigma'_v = \sigma_v$ for all but one place $v_1$ at which $\sigma'_{v_1} < \sigma_{v_1}$. As $\tau \leq \sigma' < \sigma$ we have $\tau_{v_1} \leq \sigma_{v_1}' < \sigma_{v_1}$. Since $z_{\nc}^{(\sigma)}$ satisfies $(\dagger)$ we have that $\ell(\sigma_{v_1}) - \ell(\tau_{v_1}) \leq 1$. Indeed, if $\tau_{v_1} = (1)$ then $\sigma_{v_1}$ must be a simple transposition and if $\tau_{v_1} \neq (1)$ then it follows from the classification Lemma \ref{lemm:simple-trans-non-critical}. Thus it follows that $\sigma_{v_1}' = \tau_{v_1}$. The second half of the lemma follows from the first half and the definition of $(\dagger)$.
\end{proof}

Note that by Corollary \ref{prop:non-critcal-gen-ordinary} a point $z$ satisfying $(\dagger)$ is either non-critical or possesses a unique companion point (this is the critical case). Moreover, in the latter case we have a unique companion point given by $s_{\alpha_\tau} \cdot z$ where $\tau$ is the weight-type of $z$. Note that in the case $z$ is non-critical, the weight-type is $\tau = (1)$.

\begin{prop}\label{prop:good-points}
If $z$ has weight-type $\tau$ and satisfies $(\dagger)$ then $s_{\alpha_\tau} \cdot z$ is not bad.
\end{prop}

\begin{proof}
If $z$ is non-critical then $s_{\alpha_\tau}\cdot z=z$ is not bad by Proposition \ref{prop:bad-implies-critical}. Now consider $x := s_{\alpha_\tau}\cdot z$ in the case that $z$ is critical of weight-type $\tau$. Since being strongly linked is a transitive property, we see immediately that if $x$ is bad, then $z$ has a second companion point of the form $w\cdot x = (w's_{\alpha_\tau})\cdot z$ with $w' \neq 1$. But this contradicts the uniqueness of the companion point associated to $z$ by Corollary \ref{prop:non-critcal-gen-ordinary}.
\end{proof}

We highlight separately one particular case, which could have been deduced immediately following Proposition \ref{prop:bad-implies-critical}.

\begin{coro}\label{coro:simple-pts-not-bad}
If $z \in X_{\cl}$ is $\Sigma_p$-simple, generic ordinary point and $\rho_{z,v}$ is totally indecomposable for all $v \in \Sigma_p$ then $z$ is not bad.
\end{coro}
\begin{proof}
The hypotheses on $z$ imply that $z$ is non-critical (see Corollary \ref{prop:non-critcal-gen-ordinary}) and thus not bad by Proposition \ref{prop:bad-implies-critical}.
\end{proof}

\section{An adjunction formula}\label{sec:adjunction}

Our goal in this section is to produce an adjunction formula between certain locally analytic principal series and completed cohomology. The analog for $\GL_2(\mbf Q_p)$ is given by \cite[Theorem 5.5.1]{be}. It is crucial, as in the work of Breuil and Emerton, that our result be applicable beyond the classical situation. We now state the adjunction formula from Theorem C.

\begin{theo}\label{theo:analytic-maps}
Let $z=(\chi,\lambda) \in X(L)$ be not bad. Then there exists an isomorphism
\begin{equation*}
\Hom_{G_{\Sigma _p}}((\Ind_{B _{\Sigma _p} ^{-}} ^{G _{\Sigma _p}} \chi \delta _{B_{\Sigma _p}} ^{-1})^{\an}, \widehat H^{0}(K^p)^{\lambda} _{L, \an}) \simeq \widehat H^{0}(K^p)_{L, \an} ^{N_{\Sigma _{p}} ^0,T _{\Sigma _p}^+=\chi,\lambda} \simeq J_{B_{\Sigma _p}}^{\chi}(\widehat H^{0}(K^p)^{\lambda}_{L,\an})
\end{equation*}
where we have denoted by $(\Ind_{B _{\Sigma _p} ^{-}} ^{G _{\Sigma _p}} \chi \delta _{B_{\Sigma _p}} ^{-1})^{\an}$ the locally analytic induction of $\chi \delta _{B_{\Sigma _p}} ^{-1}$ seen as a $B_{\Sigma_p} ^{-}$-representation.
\end{theo}
We note that an adjunction formula was also proven in \cite[Th\'eor\`me 4.3]{br} but that ours is different. See Remark \ref{rema:breuil}.

\subsection{Preliminaries on Verma modules}

Before giving the proof we recall necessary material on Verma modules. A useful reference for the Lie-theoretic theory is \cite{hum}. We will ignore some of subscripts $\Sigma _p$. For example, we write $\mf{g}$ for $\mf{g}_{\Sigma_p}$ and $\mf b$ for $\mf{b}_{\Sigma _p}$, etc. Recall that if $H$ is any subgroup of $G$, then a $(\mf{g}, H)$-module is a $\mf{g}$-module $V$ with a linear action of $H$ such that for any $X\in \mf{g} , h\in H, v \in V$ we have $h\cdot X \cdot v = \Ad _h(X)hv$. 

If a $\mf{g}$-module $V$ is locally $\mf{n}$-nilpotent, then it carries a canonical $(\mf{g}, N)$-structure obtained by integrating the $\mf{n}$-action. Moreover, if $V$ also has the structure of a $(\mf{g}, T)$-module, then the latter extends to a $(\mf{g}, B)$-structure.

Let $\chi \in \mc T(L)$ and write $L_\chi$ for the one-dimensional $L$-vectorspace on which $T$ acts by $\chi$. We define
\begin{equation*}
M(\chi) = \mrm U(\mf{g}) \otimes _{\mrm  U(\mf{b})} L_{\chi}
\end{equation*}
and
\begin{equation*}
M(\chi)^{\vee} = \Hom _{\mrm U(\mf{b} ^-)}(\mrm  U(\mf{g}), L_{\chi})^{\mf{n} ^{\infty}},
\end{equation*}
where $(-)^{\mf n^\infty}$ refers to the $\mf n$-nilpotent vectors. The first is an analog of the classical Verma module \cite[Part I]{hum}
\begin{equation*}
\Ver(d\chi) = \mrm U(\mf g)\otimes_{\mrm U(\mf b)} L_{d\chi}
\end{equation*}
where $d\chi: \mf t \rightarrow L$ is the derivative of $\chi$ and $L_{d\chi}$ denotes this one-dimensional representation of $\mf t$, seen as a module over $\mrm U(\mf b)$ by inflation.

We endow $M(\chi)$ (respectively, $M(\chi)^{\vee}$) with a $\mf{g}$-action by left translations (respectively, right translations) by $\mrm U(\mf{g})$. This action is locally $\mf n$-nilpotent. We let $T$ act by the adjoint action on $\mrm U(\mf{g})$ and by the character $\chi$ on $L_\chi$. This gives a structure of a $(\mf{g}, T)$-module on both $M(\chi)$ and $M(\chi)^{\vee}$, which extends canonically to a structure of a $(\mf{g}, B)$-module by the previous paragraph. Restricting to $\mrm U(\mf g)$, $M(\chi) \simeq \Ver(d\chi)$.

Since $\mrm U(\mf{g}) = \mrm U(\mf{b}^-) \oplus \mrm U(\mf{g})\mf{n}$ we can consider the map 
\begin{equation*}
\mrm U(\mf{g}) \ra \mrm U(\mf{g}) / \mrm U(\mf{g})\mf{n} \ra \mrm U(\mf{b} ^-) \stackrel{\chi}{\longrightarrow} L
\end{equation*}
as an element $v_\chi^{\vee}$ of $M(\chi)^{\vee}$. The vector $v_\chi^\vee$ is killed by $\mf{n}$ and is a $\chi$-eigenvector for the action of $T$. Therefore, there is a unique $(\mf{g}, B)$-equivariant map 
\begin{equation*}
\alpha _\chi : M(\chi) \ra M(\chi)^{\vee}
\end{equation*}
which takes $v_\chi := 1\otimes 1$ to $v^{\vee}_\chi$.

\begin{lemm} \label{lem:themapalpha}
Let $\chi \in \mc T(L)$.
\begin{enumerate}
\item The map $\alpha _{\chi}$ is the unique (up to scalar) $(\mf{g},B)$-equivariant map $M(\chi) \ra M(\chi)^{\vee}$. Its image $L(\chi)$ is an irreducible $(\mf{g},B)$-module.

\item The modules $M(\chi)$ and $M(\chi)^{\vee}$ are finite length and their simple subquotients are of the form $L(\chi ')$ where $\chi '$ is strongly linked to $\chi$.
\end{enumerate}
\end{lemm}
\begin{proof}
Each statement is well-known regarding just the $\mf g$-modules structure. As $\mf{g}$-modules, $M(\chi)$ is the Verma module $\Ver(d\chi)$, $M(\chi)^{\vee}$ is the dual Verma module $\Ver(d\chi)^{\vee}$ and $\alpha _\chi$ is the usual (unique up to a scalar) map $\Ver (d\chi) \ra \Ver (d\chi)^{\vee}$ whose image is known to be the simple $\mf{g}$-module $L(d\chi)$ of highest weight $d\chi$ \cite[Theorem 3.3(c)]{hum}. Since $\alpha _{\chi}$ is also $B$-equivariant, we see that its image is a simple $(\mf{g},B)$-module which extends $L(d\chi)$. This explains (1).

 Let us prove the second statement now. The corresponding assertion for Verma modules, that the irreducible constituents of $M(\chi)$ (and $M(\chi)^\vee$) are the simple $\mf g$-modules of highest weight $\mu$ with $\mu$ strongly linked to $d\chi$, is the famous result of Bernstein-Gelfand-Gelfrand \cite[Theorem 5.1]{hum}. It suffices to check that if $\mu '$ is strongly linked to $d\chi$ then
\begin{itemize}
\item There is a unique character $\chi'$ strongly linked to $\chi$ with $d\chi ' = \mu '$.
\item Any $\mf{g}$-map $\Ver(d\chi') \ra \Ver (d\chi)$ is in fact $T$-equivariant, that is, is a map $M(\chi ') \ra M(\chi)$.
\item Any $\mf{g}$-map $\Ver(d\chi)^{\vee} \ra \Ver(d\chi')^{\vee}$ is in fact $T$-equivariant, that is, is a map $M(\chi)^{\vee} \ra M(\chi')^{\vee}$.
\end{itemize}
Let $\mu'$ be strongly linked to $d\chi$ and define $d\gamma := \mu ' - d\chi$. Since $\mu'$ is strongly linked, $d\gamma$ is integral and hence is the derivative of an algebraic character $\gamma$. Thus $\chi ' := \gamma \chi$ is the desired character. This proves the first point.

The second two points are dual to each other, so let us just deal with the second. We claim that $T$ acts via $\chi'$ on any highest weight vector $v$ of $M(\chi)$ with weight $d\chi'$. But we know that $v$ is of the form $X v_{\chi}$ for some $X \in \mrm U(\mf{n} ^-)$ (because of $\mf{g}$-equivariance of our map). Moreover $X$ is of weight $d\gamma = d\chi' - d\chi$ for the adjoint action of $\mf{t}$. Indeed, if $H \in \mf{t}$, then $H X v_{\chi} = ([H,X]+XH)v_{\chi} = ([H,X] + d\chi(H)X)v_{\chi}$. Hence $d\chi'(H)X \cdot v_{\chi} = ([H,X]+d\chi(H)X)v_{\chi}$ and because $\mf{n}^-$ cannot kill $v_{\chi}$ we have $[H,X]=(d\chi'(H)-d\chi(H))X$, thus the claim.

Since the eigenvalues of the adjoint action on $T$ acting on $\mrm U(\mf n^-)$ are algebraic, they are determined by the differential action. Thus $X$ is automatically an eigenvector for the adjoint action of $T$ on $\mrm U(\mf n^-)$ with eigenvalue given by $\gamma = \chi'\chi^{-1}$. It follows that $Xv_\chi$ is an eigenvector for $T$ with eigenvalues given by $\chi'$, which shows $T$-equivariance of the map $M(\chi') \rightarrow M(\chi)$.
\end{proof}

For comparison with Emerton's constructions in the next subsection, we now derive another description of $M(\chi)^{\vee}$. It seems to be well-known but we could not find a good reference. We define $\mc{C}^{\pol}(N, L_{\chi})$ to be the space of $L_{\chi}$-valued polynomial functions on $N$. It carries a natural structure of $B$-module as explained after \cite[Lemma 2.5.3]{em3}. It carries also a natural action of $\mf{g}$ defined as follows. Any $f\in \mc{C}^{\pol}(N,L_{\chi})$ may be extended to a locally analytic function on the big cell $B^-N$ by putting $\tilde{f}(b^-n) := \chi (b^-)f(n)$. Since $B^-N$ is open in $G$ one can then make $X \in \mf{g}$ act by left invariant derivation on $f$, that is $Xf := \partial _X \tilde{f}_{|N}$. We then have a unique $(\mf{g},B)$-equivariant map $\beta _{\chi}: M(\chi) \ra \mc{C} ^{\pol}(N, L_{\chi})$ which takes $1\otimes 1$ to the constant function with value $1$.   

\begin{lemm}\label{lem:dual-Verma}
There is a $(\mf{g},B)$-equivariant isomorphism $\mc{C}^{\pol}(N, L_{\chi}) \simeq M(\chi)^{\vee}$ that carries the map $\beta _{\chi}$ to $\alpha _\chi$ up to a scalar.
\end{lemm}
\begin{proof}
The isomorphism of vector spaces is given by \cite[(2.5.7)]{em3} and the $T$-equivariance of this map is clear from the definitions. The unicity of $\alpha_\chi$ implies the assertion about $\beta_\chi$. (We note that in our reference, Emerton defines the $\mf{g}$-action on $\mc{C}^{\pol}(N,L_{\chi})$ {\em via} this isomorphism. However, it is easy to see from \cite[Lemma 2.5.24]{em3} that his action coincides with the one we have defined above.)
\end{proof}

\subsection{Proof of the adjunction formula}\label{subsec:adjunction-proof}

Let us now start the proof of the adjunction formula, which will occupy this entire subsection. We let $(\chi,\lambda) \in X(L)$ be not bad. We denote by $\mc{C}_c^{\lp}(N, L_\chi )$ the space of compactly supported locally $L_\chi $-valued polynomial functions on $N$. Because of the natural open immersion $N \hookrightarrow G / B^-$, we can regard $\mc{C}_c^{\lp}(N, L_\chi )$ as a $(\mf{g}, B)$-invariant subspace of $\Ind _{B^-}^{G} (\chi )^{\an}$. The inclusion of $\mc{C}_c^{\sm}(N, L_\chi )$ in $\mc{C}_c^{\lp}(N, L_\chi )$ thus induces a $(\mf{g}, B)$-equivariant map
\begin{equation} \label{eqn:map2}
\mrm U(\mf{g}) \otimes _{\mrm U(\mf{b})} \mc{C}_c^{\sm}(N, L_\chi ) \ra \mc{C}_c^{\lp}(N, L_\chi )
\end{equation}
We consider, and will explain, the following commutative diagram:
\begin{equation*}
\renewcommand{\labelstyle}{\textstyle}
\xymatrix@R=3pc@C=1pc{
\Hom _{G} (\Ind _{B^{-}} ^{G} (\chi \delta  ^{-1})^{\an}, \widehat H^{0}(K^p)_{L, \an} ^{\lambda}) \ar[r]^-{(1)} \ar[d]^-{\simeq}_-{\text{(a)}}  &  \Hom _{T} (\chi, J_{B}(\widehat H^{0}(K^p)_{L, \an} ^{\lambda})) \ar[d]^{\simeq (b)} \\
\Hom _{(\mf{g},B)}(\mc C_c^{\lp}(N,L_{\chi\delta^{-1}}), \widehat H^{0}(K^p)_{L, \an} ^{\lambda}) \ar[r] ^-{(2)} \ar[d]^-{\simeq}_-{\text{(c)}} & \Hom _{(\mf{g},B)}(M(\chi) \otimes C_c ^{\sm}(N,L _{\delta  ^{-1}}), \widehat H^{0}(K^p)_{L, \an} ^{\lambda}) \ar@{=}[d] \\
\Hom _{(\mf{g},B)}(M(\chi)^{\vee} \otimes \mc{C} _c ^{\sm}(N, L _{\delta  ^{-1}}), \widehat H^{0}(K^p)_{L, \an} ^{\lambda}) \ar[r]^{(3)} \ar[d]_-{\text{(3a)}} & \Hom _{(\mf{g},B)}(M(\chi) \otimes \mc{C}_c ^{\sm}(N,L _{\delta  ^{-1}}), \widehat H^{0}(K^p)_{L, \an} ^{\lambda}) \\
\Hom _{(\mf{g},B)}(L(\chi) \otimes \mc{C}_c ^{\sm}(N,L _{\delta  ^{-1}}), \widehat H^{0}(K^p)_{L, \an} ^{\lambda}) \ar[ur]_-{(3\text{b})}
}
\end{equation*}

We want to prove that (1) is an isomorphism. Let us now explain all the identifications and maps.

Let $\Ind _{B ^-} ^{G} (\chi \delta  ^{-1})(N)$ denote the subspace of $\Ind _{B^-} ^{G} (\chi \delta  ^{-1})$ of functions supported on $N$. By \cite[Lemma 2.4.14]{em3}, it generates $\Ind _{B ^-} ^{G}(\chi \delta  ^{-1})$ as a $G$-representation and $\mc{C}_c ^{\lp}(N,L_{\chi \delta  ^{-1}})$ is a dense $(\mf{g},B)$-subrepresentation in $\Ind _{B ^-} ^{G} (\chi \delta  ^{-1} )(N)$ by  \cite[Proposition 2.7.9]{em3}. Those are basic ingredients to conclude we have a sequence of isomorphisms
\begin{align*}
\Hom _{G} (\Ind _{B ^-} ^G (\chi \delta  ^{-1} ), \widehat H^{0}(K^p)_{L, \an} ^{\lambda}) &\simeq \Hom _{(\mf{g},B)} (\Ind _{B ^-} ^{G} (\chi \delta  ^{-1} )(N), \widehat H^{0}(K^p)_{L, \an} ^{\lambda})\\
& \simeq \Hom _{(\mf{g},B)} (\mc{C} _c ^{\lp}(N,L_{\chi \delta  ^{-1}}), \widehat H^{0}(K^p)_{L, \an} ^{\lambda}),
\end{align*}
which are proved as \cite[Theorem 4.1.5]{em3} and \cite[Theorem 4.2.18]{em3}. This gives the identification (a). Since 
\begin{equation*}
\mc C_c^{\lp}(N,L_{\chi\delta^{-1}}) \simeq \mc{C}^{\pol}(N,L_\chi ) \otimes \mc{C}_c ^{\sm}(N,L_{\delta  ^{-1}}),
\end{equation*}
the identification (c) follows from Lemma \ref{lem:dual-Verma}.

Going vertically down the right-hand side, the identification (b) results from \cite[Theorem 3.5.6]{em2} (see also \cite[(0.17)]{em3}). The map (2) is induced by the natural map of $(\mf{g},B)$-modules
\begin{equation*}
\mrm U(\mf{g}) \otimes _{\mrm U(\mf{b})} L_{\chi} \ra \mc{C}^{\pol} (N, L_{\chi } ).
\end{equation*}
The map (3) is induced from the map
\begin{equation*}
\alpha _{\chi}: M(\chi) \ra M(\chi) ^{\vee}
\end{equation*}
of Lemma \ref{lem:themapalpha}. The commutation is given by Lemma \ref{lem:dual-Verma} again. By Lemma \ref{lem:themapalpha} again, $\alpha_\chi$ factors  as
\begin{equation*}
M(\chi) \twoheadrightarrow L(\chi) \hookrightarrow M( \chi) ^{\vee},
\end{equation*}
which gives maps (3a) and (3b). The map (3b) is always injective since $L(\chi)$ is a quotient of $M(\chi)$.
\begin{rema}\label{rema:breuil}
This is a good time to discuss the interaction of what we have done and the work of Breuil \cite{br}. Assume, to put us in the context of Breuil's work, that $\chi = \theta \delta$ where $\theta$ is smooth and $\delta$ is an algebraic character of $T$. Then we can take the parabolic $P$ of \cite{br} to be the lower triangular Borel $B^-$, $M = M(\delta)$ or $M = M(\delta)^\vee$ and $\pi_B = \theta$. The identifications (a) and (c) above are  \cite[Proposition 4.2]{br} with $M = M(\delta)$. If we take $M = M(\delta)^\vee$ then the main adjunction formula \cite[Th\'eor\`eme 4.3]{br} is deduced from \cite[Proposition 4.2]{br} with $M = M(\delta)^\vee$ together with the identification (b) above. None of these identifications require a hypothesis on the pair $(\chi,\lambda)$, which is why there is no hypothesis in \cite[Th\'eor\`eme 4.3]{br}.
\end{rema}

We will now prove Theorem \ref{theo:analytic-maps} by showing that the map (3) is an isomorphism. In turn we will show separately that (3a) and (3b) are isomorphisms. To aid the reader, (3a) is an isomorphism following Corollary \ref{coro:need3b} and we show that (3b) is an isomorphism.

\begin{prop}\label{prop:badness}
Suppose that $(\chi,\lambda) \in X(L)$ is not bad. If $\chi'\neq\chi$ is strongly linked to $\chi$ then
\begin{equation*}
\Hom _{(\mf{g},B)}(L(\chi')\otimes \mc{C}_c ^{\sm}(N,L_{\delta  ^{-1}}), \widehat H^{0}(K^p)_{L, \an} ^{\lambda}) = 0.
\end{equation*}
\end{prop}
\begin{proof}
Since $\Hom$ is right exact, we have 
\begin{align*}
\Hom _{(\mf{g},B)}(L(\chi')\otimes \mc{C}_c ^{\sm}(N,L_{\delta  ^{-1}}), \widehat H^0(K^p)_{L, \an} ^{\lambda}) & \subset \Hom _{(\mf{g},B)}(M(\chi')\otimes \mc{C}_c ^{\sm}(N,L _{\delta  ^{-1}}), \widehat H^{0}(K^p)_{L, \an} ^{\lambda}) \\
&= \Hom _{B}( \mc{C}_c ^{\sm}(N,L_{\chi' \delta  ^{-1}}), \widehat H^{0}(K^p)_{L, \an} ^{\lambda}).
\end{align*}
Evaluation on the characteristic function $\mbf 1_{N^0} \in \mc C_c^{\sm}(N,L_{\chi'\delta^{-1}})$ defines an isomorphism
\begin{equation*}
\Hom _{B}( \mc{C}_c ^{\sm}(N,L_{\chi' \delta  ^{-1}}), \widehat H^{0}(K^p)_{L, \an} ^{\lambda}) \simeq \widehat H ^{0}(K^p) ^{N^0, T^+ = \chi', \lambda} _{L, \an}.
\end{equation*}
However, the space on the right hand side vanishes since $\chi' \neq \chi$ is strongly linked to $\chi$ and $(\chi,\lambda)$ is not bad. (We remark that we normalize the action of $T^+$ as in \cite[Definition 3.4.2]{em2}, which explains the twists by $\delta$).
\end{proof}

\begin{coro} \label{coro:need3b}
The map (3b) is an isomorphism and the map (3a) is injective.
\end{coro}
\begin{proof}
By Lemma \ref{lem:themapalpha} the constituents of $M(\chi)$ and $M(\chi)^\vee$ are the $L(\chi')$ with $\chi'$ strongly linked to $\chi$. By Proposition \ref{prop:badness} we have that every constituent $L(\chi')\neq L( \chi)$ satisfies 
\begin{equation*}
\Hom _{(\mf{g},B)}(L(\chi')\otimes \mc{C}_c ^{sm}(N,L_{\delta  ^{-1}}), \widehat H^{0}(K^p)_{L, \an} ^{\lambda}) = 0.
\end{equation*}
Thus we clearly see that (3a) is injective and (3b) is surjective. Since the map (3b) is always injective, this proves that (3b) is an isomorphism.
\end{proof}
To finish the proof of the adjunction formula we need to prove that (3a) is surjective.  Let us denote by $\mf{p} _{\lambda}$ the ideal of the Hecke algebra $\mc{H}(K^p)^{\nr}$ corresponding to $\lambda$ and by $\mf{p} ^+ _{\id}$ the ideal of $L[T^+ ]$ corresponding to the character $id$.
\begin{lemm}
Let $V$ and $W$ be $(\mf{g},B)$-modules. We have
\begin{equation*}
\Hom _{(\mf{g},B)}(V \otimes \mc{C}_c ^{\sm}(N, L_{\delta ^{-1}}),W) = \Hom _{\mf{g}}(V,W) ^{N^0, T^+=\id}.
\end{equation*}
\end{lemm}
We remark again that the twist by $\delta $ in the $T^+$-action appears because we use the normalization of Emerton from \cite[Definition 3.4.2]{em2} in defining the action of the monoid $T^+$ on the $N^0$-invariants. If we were to make $T^+$ act by correspondences $[N   ^0 t^+ N   ^0]$, then we would have $T^+ = \delta ^{-1}$ on the right.
\begin{proof}
Both sides are isomorphic to $\Hom _{B}(\mc{C}_c ^{\sm}(N, L_{\delta ^{-1}}),\Hom_{\mf{g}}(V,W))$, where we pass to the right-hand side by evaluation at the characteristic function $1_{N^0}$.
\end{proof}
We consider the four-term exact sequence
\begin{multline*}
0 \ra \Hom _{\mf{g}}(M(\chi)^{\vee} / L(\chi) , \widehat H^{0}(K^p)_{L, \an})^{N^0} \ra \Hom _{\mf{g}}(M(\chi)^{\vee}  , \widehat H^{0}(K^p)_{L, \an} )^{N^0} \ra \\
\Hom _{\mf{g}}( L(\chi) , \widehat H^{0}(K^p)_{L, \an} )^{N^0} \ra \Ext ^1 _{\mf{g}}(M(\chi)^{\vee} / L(\chi) , \widehat H^{0}(K^p)_{L, \an})^{N^0}.
\end{multline*}
We want to prove that
$$\Hom _{\mf{g}}(M(\chi)^{\vee} , \widehat H^{0}(K^p)_{L, \an} )^{N^0}[\mf{p} _{\lambda}, \mf{p}^+ _{\id}] \simeq \Hom _{\mf{g}}( L(\chi), \widehat H^{0}(K^p)_{L, \an} )^{N^0} [\mf{p} _{\lambda}, \mf{p}^+ _{\id}]$$
Because of the exact sequence above it is enough to show
$$\Hom _{\mf{g}}(M(\chi)^{\vee} / L(\chi) , \widehat H^{0}(K^p)_{L, \an})^{N^0} _{\mf{p} _{\lambda}, \mf{p}^+ _{\id}}=0$$ 
$$\Ext ^1 _{\mf{g}}(M(\chi)^{\vee} / L(\chi) , \widehat H^{0}(K^p)_{L, \an})^{N^0} _{\mf{p} _{\lambda}, \mf{p}^+ _{\id}} = 0$$
where the subscript $\mf{p} _{\lambda}, \mf{p}^+ _{\id}$ denotes the localisation at respective ideals. By devissage, it suffices to show

\begin{lemm}\label{lem:localisation} Suppose that $(\chi,\lambda)$ is not bad, $\chi'$ is strongly linked to $\chi$ and let $\mu = d\chi'$. Then
\begin{enumerate}
\item[(a)]  $\Hom _{\mf{g}}(L(\mu), \widehat H^{0}(K^p)_{L, \an} ^{\lambda}) _{\mf{p} _{\lambda}, \mf{p} ^+ _{\id }} ^{N^0} =0$ and
\item[(b)] $\Ext ^1 _{\mf{g}}(L(\mu) , \widehat H^{0}(K^p)_{L, \an})^{N^0} _{\mf{p} _{\lambda}, \mf{p}^+ _{\id}}  = 0.$
\end{enumerate}
\end{lemm}
\begin{proof}
(a) We can replace $L(\mu)$ by $M(\mu)$. Thus it is enough to prove that 
$$\Hom _{\mf{g}}(M(\mu), \widehat H^{0}(K^p)_{L, \an} ) _{\mf{p} _{\lambda}, \mf{p} ^+ _{\id }} ^{N^0} = 0 $$
But $\Hom _{\mf{g}}(M(\mu), \widehat H^{0}(K^p)_{L, \an} ) ^{N^0} = \widehat H^{0}(K^p)_{L, \an} ^{N^0, \mf{t} = \mu}$. It follows from \cite[Lemma 2.3.4(ii)]{em2} that $\mc{H}(K^p)^{\nr}$ and $T^+$ act by compact operators on $\widehat H^{0}(K^p)_{L, \an} ^{N^0, \mf{t} = \mu}$ because this space is of compact type. Hence the localisation $(\widehat H ^0 (K^p)^{N^0, \mf{t} = \mu} _{L, \an})_{\mf{p}_{\lambda}, \mf{p} ^+ _{\id }}$ is finite dimensional. Therefore 
$$\Hom _{\mf{g}}(M(\mu), \widehat H^{0}(K^p)_{L, \an} ) _{\mf{p} _{\lambda}, \mf{p} ^+ _{\id }} ^{N^0}$$ 
is of $\mf{p} _{\lambda}$-torsion and of $\mf{p} ^+ _{id}$-torsion. Hence if this space is non-zero then also
$$\Hom _{\mf{g}}(M(\mu), \widehat H^{0}(K^p)_{L, \an} )^{N^0}[\mf{p} _{\lambda}, \mf{p} ^+ _{\id} ] = \widehat{H} ^{0}(K^p)_{L, \an} ^{N  ^0, T^+ = \mu, \lambda}$$
is non-zero which would contradict our assumption that $(\chi,\lambda)$ is not bad. Hence we conclude.

(b)  Let us prove firstly that $\Ext ^1 _{\mf{g}}(L(\mu) , \widehat H^{0}(K^p)_{L, \an})^{N^0} _{\mf{p} _{\lambda}, \mf{p}^+ _{\id}}$ is of $\mf{p} _{\lambda}$-torsion and of $\mf{p} ^+ _{\id}$-torsion. By the exact sequence
$$\Hom _{\mf{g}}(\ker(M(\mu) \ra L(\mu)), \widehat H^{0}(K^p)_{L, \an}) ^{N^0} \ra \Ext ^1 _{\mf{g}}(L(\mu) , \widehat H^{0}(K^p)_{L, \an})^{N^0}  \ra$$
$$\ra \Ext ^1 _{\mf{g}}(M(\mu), \widehat H^{0}(K^p)_{L, \an})^{N^0} $$
and devissage from (a), it suffices to prove that $\Ext ^1 _{\mf{g}}(M(\mu), \widehat H^{0}(K^p)_{L, \an}) ^{N^0} _{\mf{p} _{\lambda}, \mf{p}^+ _{\id}} = 0$. We have 
$$\Ext ^1 _{\mf{g}}(M(\mu), \widehat H^{0}(K^p)_{L, \an}) = \Ext ^1 _{\mf{b}}(\mu, \widehat H^{0}(K^p)_{L, \an}) = H^1(\mf{b}, \widehat H^{0}(K^p)_{L, \an} (-\mu)).$$
We know that $\widehat H^{0}(K^p)_{L, \an}$ is injective as a $G(\mbf{Z}_p)$-module. The standard argument is given for example in the proof of Proposition 4.9 in \cite{cho3}, from which we infer that in fact $\widehat H^{0}(K^p)_{L, \an} \simeq \mc{C}^{\la} (G(\mbf{Z}_p),L)^{\oplus r}$ as $G(\mbf{Z}_p)$-modules, for some integer $r>0$ . Now, observe that $\mc{C}^{\la} (G(\mbf{Z}_p),L)$ is $\mf{b}$-acyclic and invariant under $\mu$-torsion (it is stated in the proof of \cite[Proposition 5.1.2]{be}, compare also with the proof of \cite[Proposition 3.1]{st4}, where $\mf{g}$-acyclicity is proved), hence  
$$H^1(\mf{b}, \widehat H^{0}(K^p)_{L, \an} (-\mu))=0.$$
This proves the claim. 
\end{proof}

This finishes the proof of the adjunction formula.

\subsection{Principal series on the generic ordinary locus}

Due to the normalizations in the rest of our work we will now write points $z = (\chi, \lambda) \in X(\overline{\mbf Q}_p)$ as pairs $(\xi^{(13)}\delta_{B_{\Sigma _p}},\lambda)$ where $\xi \in \mc T_{\Sigma_p}$, $\xi^{(13)}$ is the usual twisting by the longest Weyl-element $(13)$:
\begin{equation*}
(\xi_1\otimes \xi_2\otimes \xi_3)^{(13)} = \xi_{3}\otimes\xi_2\otimes \xi_1
\end{equation*}
and $\delta _{B_{\Sigma _p}} = (\delta _{B_{v}})_{v\in \Sigma _p}$ is the $\Sigma_p$-tuple of modulus characters. This allows us to pass between Borel $B_{\Sigma _p}$ and its opposite $B _{\Sigma _p} ^-$:
$$(\Ind_{B _{\Sigma _p}} ^{G _{\Sigma _p}} \xi)^{\an} = (\Ind_{B _{\Sigma _p} ^-} ^{G _{\Sigma _p}} \chi \delta _{B_{\Sigma _p}} ^{-1})^{\an}$$
We will use that tacitly in the rest of the text.

We finish this subsection by giving our main application of Theorem \ref{theo:analytic-maps} on the generic ordinary locus described in Section \ref{subsec:generic-ordinary}. If $z = z_{\nc}^{(\sigma)}$ is a generic ordinary point then $z$ is of the form $z = (\theta^{\nc,(\sigma)} \delta_k, \lambda_z)$ where
\begin{equation*}
\theta_v^{\nc, (\sigma_v)} = (\psi_{\sigma_v(1)}z^{h_{\sigma_v(1)}} \otimes \psi_{\sigma_v(2)}z^{h_{\sigma_v(2)}} \otimes \psi_{\sigma_v(3)}z^{h_{\sigma_v(3)}})(|\cdot|^2 \otimes |\cdot| \otimes 1)
\end{equation*}
and $k_v = (k_{1,v} \geq k_{2,v} \geq k_{3,_v})$ is a dominant weight for $\GL_3(\mbf Q_p)$. It will be convenient to use the Hodge-Tate weights $h_{i,v} = -k_{i,v} + i -1$. If we then write $z$ in the form $z = z_{\nc}^{(\sigma)} = ((\xi_{z_{\nc}^{(\sigma)}})^{(13)} \delta_{B_{\Sigma _p}},\lambda)$, we see that
\begin{equation*}
\xi_{z_{\nc}^{(\sigma)}} = \bigotimes_{v \in \Sigma_p} \left(\psi_{\sigma_v(3),v}z^{h_{\sigma_v(3),v}-h_{3,v}}\varepsilon^2 \otimes \psi_{\sigma_v(2),v}z^{h_{\sigma_v(2),v}-h_{2,v}}\varepsilon \otimes \psi_{\sigma_v(1)}z^{h_{\sigma_v(1),v}-h_{1,v}}\right)
\end{equation*}
As Theorem \ref{theo:analytic-maps} only applies to not bad points $z$, we will need to compute slightly more. Let $\tau$ be the weight-type of a point $z$ satisfying $(\dagger)$. Then, by Proposition \ref{prop:good-points} we have that 
\begin{equation*}
s_{\alpha_\tau} \cdot z = (\theta_{\nc}^{(\sigma)}\delta_{\tau\cdot k},\lambda_z)
\end{equation*}
is not bad. A short computation shows that $s_{\alpha_\tau}\cdot z = ((\xi_{s_{\alpha_\tau}\cdot z})^{(13)}\delta_{B_{\Sigma_p}},\lambda_z)$ where
\begin{equation}\label{eqn:char-of-xgood}
\xi_{s_{\alpha_\tau}\cdot z} := \bigotimes_{v \in \Sigma_p} \left(\psi_{\sigma_v(3),v}z^{h_{\sigma_v(3),v}-h_{\tau_v(3),v}}\varepsilon^2 \otimes \psi_{\sigma_v(2),v}z^{h_{\sigma_v(2),v}-h_{\tau_v(2),v}}\varepsilon \otimes \psi_{\sigma_v(1)}z^{h_{\sigma_v(1),v}-h_{\tau_v(1),v}}\right).
\end{equation}
Continue to assume that $z$ satisfies $(\dagger)$ and define a locally analytic (respectively, continuous) principal series representation
\begin{equation}\label{eqn:def-E}
I_{\ast}(z) = (\Ind _{B_{\Sigma _p}} ^{G_{\Sigma _p}} \xi_{s_{\alpha_\tau}\cdot z})^{\ast}
\end{equation}
where $\ast \in \{\an,C^0\}$. Notice that if $\xi_{s_{\alpha_\tau}\cdot z}$ is a unitary character of $\mc T_{\Sigma_p}$ (which happens if and only if $\sigma_v(i) = \tau_v(i)$ for all $i$ and $v \in \Sigma_p)$ then the universal unitary completion $I_{\an}(z)^\wedge$ of $I_{\an}(z)$ is $I_{C^0}(z)$. Our main corollary of Theorem \ref{theo:analytic-maps} is the following.
\begin{coro}\label{coro:dagger-map}
Suppose that $z=(\chi,\lambda) \in X_{\cl}$ satisfies $(\dagger)$. Then there exists a non-zero map
\begin{equation*}
I_{\an}(z)^\wedge \ra \widehat{H}^0(K^p)_{L} ^{\lambda}
\end{equation*}
\end{coro}
\begin{proof}
This follows from Theorem \ref{theo:analytic-maps}, which is valid by Proposition \ref{prop:good-points}.
\end{proof}
We will see later (see Proposition \ref{prop:multiplicity}(2)) that any such map is injective. This being granted, we can say that for any point $z$ satisfying $(\dagger)$ we have realized a locally analytic principal series corresponding to $z$ inside of the $G_{\Sigma_p}$-representation $\widehat H^0(K^p)^{\lambda}_{L}$. 

\section{Application to the conjectures of Breuil and Herzig}
\label{sec:breuil-herzig-section}

\subsection{Unitary completions of locally analytic representations}\label{subsec:unit-comp}

We now relate our work to certain representations which arise in the work of Breuil-Herzig. We will first describe a certain locally analytic representation $\Pi(\rho_p)^{\ord}$ of $\GL_3(\mbf Q_p)$ associated to a indecomposable generic ordinary representation $\rho_p$ by Breuil and Herzig. Then we show that it arises as a $\GL_3(\mbf Q_p)$-subrepresentation of a completed cohomology space. For brevity and clarity, we will explicitly describe all the representations that we need; for the combinatorics of the general situation, see \cite{bh}.

For each simple root $\alpha=\alpha_1$ or $\alpha_2$ for $\GL_3(\mbf Q_p)$, positive with respect to the upper triangular Borel $B(\mbf Q_p)$, we get a Levi component $G_{\alpha}$ of a parabolic $P_\alpha$ such that $\alpha$ is the unique simple positive root appearing in $B(\mbf Q_p) \cap G_\alpha$. For example, if $\alpha = \alpha_1$ then
\begin{equation*}
B(\mbf Q_p) \cap G_{\alpha} = \begin{pmatrix} * & *\\
& *\\
& & *\end{pmatrix} \subset G_{\alpha} = \begin{pmatrix} * & * \\ * & *\\ & & *\end{pmatrix} \subset P_{\alpha} =\begin{pmatrix} * & * & * \\ * & * & *\\ & & *\end{pmatrix}.
\end{equation*}

Consider a pair of continuous unitary characters $\psi, \psi': \mbf Q_p^\times \rightarrow \mc O_L^\times$. If $\psi\psi'^{-1}$ is not of the form $\varepsilon^{\pm 1}$ then \cite[Proposition B.2]{bh} implies that there is a unique non-split extension
\begin{equation*}
0\rightarrow \Ind_{\left(\begin{smallmatrix} * & * \\ & * \end{smallmatrix}\right)}^{\GL_2(\mbf Q_p)}(\psi' \varepsilon \otimes \psi)^{C^0} \rightarrow E(\psi,\psi') \rightarrow \Ind_{\left(\begin{smallmatrix} * & * \\ & * \end{smallmatrix}\right)}^{\GL_2(\mbf Q_p)}(\psi \varepsilon \otimes \psi')^{C^0} \rightarrow 0.
\end{equation*}
Moreover, if $\psi'\psi^{-1}$ has a positive weight Hodge-Tate weight $h$ then \cite[Th\'eor\`eme 2.2.2]{be} implies that
\begin{equation}\label{eqn:bh-realization}
E(\psi,\psi') \simeq \Ind_{\left(\begin{smallmatrix} * & * \\ & * \end{smallmatrix}\right)}^{\GL_2(\mbf Q_p)}(\psi \varepsilon z^{-h} \otimes \psi z^h)^{\an,\wedge}.
\end{equation}
See also \cite[\S 6.3]{em4} and note that both references use Hodge-Tate weights which are the negative of ours).

If $\psi'': \mbf Q_p^\times \rightarrow \mc O_L^\times$ is another continuous unitary character we consider the representation $E(\psi,\psi') \otimes \psi''$ of $G_\alpha$ for any simple positive root $\alpha$ of $\GL_3(\mbf Q_p)$, and form the continuous induction $\Ind_{P_\alpha}^{\GL_3(\mbf Q_p)}(E(\psi,\psi')\otimes \psi'')^{C^0}$. By definition these give non-zero extension classes
\begin{equation}\label{eqn:ext-class23}
0 \rightarrow \Ind_{B(\mbf Q_p)}^{\GL_3(\mbf Q_p)}(\psi'' \otimes \psi'\varepsilon \otimes \psi)^{C^0} \rightarrow \Ind_{P_{\alpha_2}}^{\GL_3(\mbf Q_p)}(\psi''\otimes E(\psi,\psi'))^{C^0} \rightarrow \Ind_{B(\mbf Q_p)}^{\GL_3(\mbf Q_p)}(\psi'' \otimes \psi\varepsilon \otimes \psi')^{C^0} \rightarrow 0.
\end{equation}
and
\begin{equation}\label{eqn:ext-class12}
0 \rightarrow \Ind_{B(\mbf Q_p)}^{\GL_3(\mbf Q_p)}( \psi'\varepsilon \otimes \psi \otimes \psi'')^{C^0} \rightarrow \Ind_{P_{\alpha_1}}^{\GL_3(\mbf Q_p)}(E(\psi,\psi')\otimes \psi'')^{C^0} \rightarrow \Ind_{B(\mbf Q_p)}^{\GL_3(\mbf Q_p)}( \psi\varepsilon \otimes \psi' \otimes \psi'')^{C^0} \rightarrow 0.
\end{equation}

Now suppose that $\rho_p$ is a generic ordinary, crystalline representation of $G_{\mbf Q_p}$ with Hodge-Tate weights $h_1 < h_2 < h_3$ and write
\begin{equation}\label{eqn:rhop-in-this-sec}
\rho_p \sim \begin{pmatrix}
\psi_1 & * & *\\
& \psi_2 & *\\
& & \psi_3
\end{pmatrix}
\end{equation}
with $\psi_i$ having weight $h_i$.  Given an element $\sigma \in S_3$ we consider
\begin{equation}\label{eqn:cont-definition}
I(\rho_p,\sigma) := \Ind_{B(\mbf Q_p)}^{\GL_3(\mbf Q_p)}(\psi_{\sigma(3)}\varepsilon^{2} \otimes \psi_{\sigma(2)}\varepsilon\otimes \psi_{\sigma(1)})^{C^0}.
\end{equation}
The unfortunate indexing is due to our choice of working with the upper triangular Borel, rather than the lower triangular one as in \cite{bh}. 

The generic hypothesis on $\rho_p$ implies that our previous discussion, in particular \eqref{eqn:ext-class23} and \eqref{eqn:ext-class12},  can be applied with $(\psi,\psi',\psi'')$ to be any ordering of the three characters $\psi_j$. Thus we have two non-split extensions
\begin{align}
0 \rightarrow I(\rho_p,\sigma(12)) \rightarrow &\Ind_{P_{\alpha_2}}^{\GL_3(\mbf Q_p)}(\psi_{\sigma(3)}\varepsilon^2 \otimes E(\psi_{\sigma(2)},\psi_{\sigma(1)}))^{C^0} \rightarrow I(\rho_p,\sigma) \rightarrow 0 \label{eqn:e2-e3-ext}\\
0 \rightarrow I(\rho_p,\sigma(23)) \rightarrow &\Ind_{P_{\alpha_1}}^{\GL_3(\mbf Q_p)}(E(\psi_{\sigma(3)}\varepsilon,\psi_{\sigma(2)}\varepsilon)\otimes \psi_{\sigma(1)})^{C^0} \rightarrow I(\rho_p,\sigma) \rightarrow 0  \label{eqn:e1-e2-ext}
\end{align}
for all $\sigma \in S_3$. To shorten notation, for each $\sigma \in S_3$, we define the middle terms as
\begin{align}
\Pi(\rho_p,\sigma)_{\alpha_2} &:= \Ind_{P_{\alpha_2}}^{\GL_3(\mbf Q_p)}(\psi_{\sigma(3)}\varepsilon^2 \otimes E(\psi_{\sigma(2)},\psi_{\sigma(1)}))^{C^0} \label{eqn:defn23}\\
\Pi(\rho_p,\sigma)_{\alpha_1} &:= \Ind_{P_{\alpha_1}}^{\GL_3(\mbf Q_p)}(E(\psi_{\sigma(3)}\varepsilon,\psi_{\sigma(2)}\varepsilon)\otimes \psi_{\sigma(1)})^{C^0} \label{eqn:defn12}.
\end{align}
We now define the representation $\Pi(\rho_p)^{\ord}$ of \cite{bh}. Notice that $\Pi(\rho_p,(12))_{\alpha_2}$ and $\Pi(\rho_p,(23))_{\alpha_1}$ both contain the common subrepresentation $I(\rho_p,(1))$ and thus we can form their amalgamated product $\Pi(\rho_p,(12))_{\alpha_2} \times_{I(\rho_p,(1))} \Pi(\rho_p,(23))_{\alpha_1}$. 
\begin{defi}\label{defi:ord}
If $\rho_p$ is generic ordinary as in \eqref{eqn:rhop-in-this-sec}, and indecomposable, then we define
\begin{equation*}
\Pi(\rho_p)^{\ord} := \begin{cases}
\Pi(\rho_p,(12))_{\alpha_2} \oplus_{I(\rho_p,(1))} \Pi(\rho_p,(23))_{\alpha_1} & \text{if $\rho_p$ is totally indecomposable},\\
 \Pi(\rho_p,(123))_{\alpha_1} \oplus \Pi(\rho_p,(23))_{\alpha_1} & \text{if $\Hom(\psi_2,\rho_p) \neq 0$},\\
 \Pi(\rho_p,(132))_{\alpha_2} \oplus \Pi(\rho_p,(12))_{\alpha_2} & \text{if $\Hom(\psi_3, \rho_p/\psi_1) \neq 0$}.
\end{cases}
\end{equation*}
\end{defi}
For comparision with \cite{bh} it is worth mentioning that each of these representations has a two-step filtration $0 \subsetneq \Fil^0\Pi(\rho_p)^{\ord} \subsetneq \Pi(\rho_p)^{\ord}$ whose associated gradeds are
\begin{equation*}
\Fil^0 \Pi(\rho_p)^{\ord} = \begin{cases}
I(\rho_p,(1)) & \text{if $\rho_p$ is totally indecomposable},\\
I(\rho_p,(12)) \oplus I(\rho_p,(1)) & \text{if $\Hom(\psi_2,\rho) \neq 0$},\\
I(\rho_p,(23)) \oplus I(\rho_p,(1)) & \text{if $\Hom(\psi_3,\rho_p/\psi_1)) \neq 0$}
\end{cases}
\end{equation*}
and
\begin{equation*}
\Pi(\rho_p)^{\ord}/\Fil^0 \Pi(\rho_p)^{\ord} = \begin{cases}
I(\rho_p,(12)) \oplus I(\rho_p,(23)) & \text{if $\rho_p$ is totally indecomposable},\\
I(\rho_p,(123)) \oplus I(\rho_p,(23)) & \text{if $\Hom(\psi_2,\rho) \neq 0$},\\
I(\rho_p,(132)) \oplus I(\rho_p,(12)) & \text{if $\Hom(\psi_3,\rho_p/\psi_1)) \neq 0$.}
\end{cases}
\end{equation*}
This completes our recollection of the representations $\Pi(\rho_p)^{\ord}$ from \cite{bh}. We next obtain an explicit description of the representations $\Pi(\rho_p,\sigma)_\alpha$ as universal unitary completions of locally analytic inductions as in \eqref{eqn:bh-realization}. We begin with a short lemma.
\begin{lemm}\label{lem:dense}
Let $P \subset G$ be a parabolic subgroup of a p-adic reductive group $G$. Let $(V_{\an}, \pi_{\an})$ be a locally analytic $L$-representation of $P$ and let $(V,\pi)$ be its universal unitary completion, which we assume to be non-zero. Suppose also that we have a $P$-equivariant injection $V_{\an} \hookrightarrow V$ (which is necessarily dense). Then we have a dense $P$-equivariant injection
$$\Ind _{P} ^G (V_{\an})^{\an} \hookrightarrow \Ind _{P} ^G (V)^{C^0}.$$  
\end{lemm}
\begin{proof}
Let $f\in \Ind _{P} ^G (V)^{C^0}$. We want to approximate it by functions in $\Ind _{P} ^G (V_{\an})^{\an}$. So let us take any $\epsilon >0$. Remark that we have $P \backslash G \simeq P_0 \backslash G_0$ where $P_0 \subset G_0$ are respectively maximal compact subgroups of $P,G$. Let $\overline{U}$ be the opposite unipotent and let $H$ be any sufficiently small compact open subgroup of $G$. We have a decomposition $G_0 = \coprod _{g_0} P_0 \times (H \cap \overline{U}) \times g_0$ into finitely many pieces. Let us choose an element $v_0 \in V_{\an}$ such that for any $h \in H \cap \bar{U}$ we have $|f(hg_0)-v_0| < \epsilon$. This is possible by density of $V_{\an}$ in $V$ and choosing $H$ sufficiently small ($H$ will depend on $f$). Then we define $f_{\an}(phg_0) = \pi_{\an}(p)v_0$ for $p \in P$ and $h \in H \cap \overline{U}$. We make a similar definition for other pieces in the decomposition hence obtaining a function $f_{\an}: G \ra V_{\an}$. Observe that because $\pi$ is unitary, $\pi$ is bounded on $P$ and hence $\sup _{p\in P} |\pi(p)|$ is finite. We have
\begin{equation*}
|f(phg_0) - f_{\an}(phg_0)|=|\pi(p)|\cdot |f(hg_0)-v_0| < \epsilon \cdot \sup _{p\in P} |\pi (p)|,
\end{equation*}
which allow us to conclude, as $\epsilon$ was arbitrary.
\end{proof}
Following the lemma, we can provide the following explicit descriptions of the representations $\Pi(\rho_p,\sigma)_\alpha$ in certain cases.
\begin{prop}\label{prop:identification}
Let $\rho_p$ be as in \eqref{eqn:rhop-in-this-sec} and $\sigma \in S_3$.
\begin{itemize}
\item If $\sigma(1) > \sigma(2)$ then 
\begin{equation*}
\Pi(\rho_p,\sigma)_{\alpha_2} \simeq \Ind_{B(\mbf Q_p)}^{\GL_3(\mbf Q_p)}(\psi_{\sigma(3)}\varepsilon^2 \otimes \psi_{\sigma(2)}\varepsilon z^{h_{\sigma(2)}-h_{\sigma(1)}} \otimes \psi_{\sigma(1)}z^{h_{\sigma(1)}-h_{\sigma(2)}})^{\an,\wedge}.
\end{equation*}
\item If $\sigma(2) > \sigma(3)$ then 
\begin{equation*}
\Pi(\rho_p,\sigma)_{\alpha_1} \simeq \Ind_{B(\mbf Q_p)}^{\GL_3(\mbf Q_p)}(\psi_{\sigma(3)}\varepsilon^2 z^{h_{\sigma(3)}-h_{\sigma(2)}}\otimes \psi_{\sigma(2)}\varepsilon z^{h_{\sigma(2)}-h_{\sigma(3)}} \otimes \psi_{\sigma(1)})^{\an,\wedge}.
\end{equation*}
\end{itemize}
\end{prop}
\begin{proof}
The proofs are symmetric so we only cover the first isomorphism. Assume that $\sigma(1) > \sigma(2)$. We write
\begin{align*}
I &= \Ind_{B(\mbf Q_p)}^{\GL_3(\mbf Q_p)}(\psi_{\sigma(3)}\varepsilon^2 \otimes \psi_{\sigma(2)}\varepsilon z^{h_{\sigma(2)}-h_{\sigma(1)}}\otimes \psi_{\sigma(1)}z^{h_{\sigma(1)}-h_{\sigma(2)}})^{\an}\\
 &\simeq \Ind_{P_{\alpha_2}}^{\GL_3(\mbf Q_p)}\left(\psi_{\sigma(3)}\varepsilon^2 \otimes \Ind_{\left(\begin{smallmatrix} * & * \\ & * \end{smallmatrix}\right)}^{\GL_2(\mbf Q_p)}(\psi_{\sigma(2)}\varepsilon z^{h_{\sigma(2)}-h_{\sigma(1)}} \otimes  \psi_{\sigma(1)}z^{h_{\sigma(1)}-h_{\sigma(2)}})^{\an}\right)^{\an},
\end{align*}
the tensor representation in the second line being seen as a representation of $P_{\alpha_2}$ via inflation $P_{\alpha_2} \rightarrow G_{\alpha_2}$. Since $\sigma(1) > \sigma(2)$ the character $\psi_{\sigma(1)}\psi_{\sigma(2)}^{-1}$ has Hodge-Tate weight $h_{\sigma(1)} - h_{\sigma(2)} > 0$. Thus we combine \eqref{eqn:bh-realization}, \eqref{eqn:defn23} and Lemma \ref{lem:dense} to deduce that there is a dense inclusion $I \hookrightarrow \Pi(\rho_p,\sigma)_{\alpha_2}$. Passing to the universal unitary completion induces a non-zero map $I^{\wedge} \rightarrow \Pi(\rho_p,\sigma)_{\alpha_2}$. Since both representations are admissible, to show that it is an isomorphism we only need to show that it is injective.

Since $I$ is dense in $\Pi(\rho_p,\sigma)_{\alpha_2}$ we can pull back the unit ball and obtain an open, separated, $\GL_3(\mbf Q_p)$-stable lattice $\Lambda_1 \subset I$ such that $\Pi(\rho_p,\sigma)_{\alpha_2} \simeq \varprojlim_n \Lambda_1/p^n\Lambda_1 \otimes \mbf Q_p$. Since $I^{\wedge}\neq 0$ we also obtain a minimal $\GL_3(\mbf Q_p)$-stable lattice $\Lambda_0$ such that $I^\wedge \simeq \varprojlim_n \Lambda_0/p^n\Lambda_0 \otimes \mbf Q_p$. By minimality $\Lambda_0 \subset \Lambda_1$. Since each is separated, we have
\begin{equation*}
I^\wedge \simeq \varprojlim_n \Lambda_0/p^n\Lambda_0 \otimes_{\mbf Z_p} \mbf Q_p \hookrightarrow \varprojlim_n \Lambda_1/p^n\Lambda_1 \otimes_{\mbf Z_p} \mbf Q_p \simeq \Pi(\rho_p,\sigma)_{\alpha_2}.
\end{equation*}
This completes the proof.
\end{proof}

\subsection{On a conjecture of Breuil and Herzig}\label{subsec:breuil-herzig}

We now go back to our global setting. We suppose that $z = (\chi_z,\lambda_z) \in X(\overline{\mbf Q}_p)$ is a generic ordinary point such that $\rho_{z,v}$ is indecomposable for each place $v \in \Sigma_p$. Using the previous section we can define a representation
\begin{equation*}
\Pi(\rho_{z,p})^{\ord} := \widehat{\bigotimes}_{v \in \Sigma_p}\Pi(\rho_{z,v})^{\ord}.
\end{equation*}
Note that this only depends only on $\rho_{z}$, not on $z$ itself (i.e. on $\lambda_z$, not $\chi_z$). We use $\widehat H^0(K^p)^{\lambda_z,\ord}_L$ to denote the ordinary part of the $G_{\Sigma_p}$-representation $\widehat H^0(K^p)^{\lambda_z}_L$. Breuil and Herzig have made the following conjecture\footnote{The conjecture of \cite{bh} is more general, dealing with any generic ordinary situation, but we have only defined the representations $\Pi(\rho_{z,p})^{\ord}$ in the case that $\rho_z$ is indecomposable above $p$.} regarding this subspace.

\begin{conj}[{\cite[Conjecture 4.2.2]{bh}}]\label{conj:bh}
Suppose that $z \in X(L)$ is a generic ordinary point and $\rho_{z,v}$ is  indecomposable at each place $v \in \Sigma_p$. Then there exists an integer $d\geq 1$ and a $G_{\Sigma_p}$-equivariant isomorphism 
\begin{equation*}
\left(\Pi(\rho_{z,p})^{\ord}\right)^{\oplus d} \simeq \widehat H^0(K^p)^{\lambda_z,\ord}_{L}.
\end{equation*}
\end{conj}
We further remark that the integer $d$ should only depend on $K^p$ and $\lambda_z$ (e.g. because the right hand side depends only on those things).

There are two differences between our normalizations and those of \cite{bh}. First, we have used the upper triangular Borels throughout. Second, the representation defined in \cite[\S 3]{bh}, temporarily denoted by $\Pi(\rho_p)^{\ord,\operatorname{BH}}$, differs from ours by the equation
\begin{equation*}
\Pi(\rho_p)^{\ord,\operatorname{BH}} \otimes \varepsilon^2 \circ \det = \Pi(\rho_p)^{\ord}.
\end{equation*}
With these modifications pointed out, one can easily see that what we have written is the same as \cite[Conjecture 4.2.2]{bh}. The following two theorems are strong evidence for their conjecture.
\begin{theo}\label{thm:we-did-it}
Suppose that $z \in X_{\cl}(L)$ is a generic ordinary point and $\rho_{z,v}$ is totally indecomposable at each place $v \in \Sigma_p$. Then there exists a $G_{\Sigma_p}$-equivariant closed embedding
\begin{equation*}
\Pi(\rho_{z,p})^{\ord} \hookrightarrow \widehat H^0(K^p)_L^{\lambda_z,\ord}.
\end{equation*}
\end{theo}
We note that this result was independently obtained by Breuil and Herzig in \cite[Theorem 4.4.7]{bh} (following the original release of their preprint) under the additional assumption that the reduction mod $p$ of the representation $\rho_z$ is also totally indecomposable above $p$. Their method of proof is completely different than ours. The following result is not handled under any hypotheses in \cite{bh}.

\begin{theo}\label{theo:F^+=Q}
Suppose that $F^+ = \mbf Q$ and that $z \in X_{\cl}(L)$ is a generic ordinary point such that $\rho_{z,v}$ is indecomposable, but not necessarily totally indecomposable for the unique $p$-adic place $v$. Then, there exists a $\GL_3(\mbf Q_p)$-equivariant closed embedding 
\begin{equation*}
\Pi(\rho_{z,v})^{\ord} \hookrightarrow \widehat H^0(K^p)_L^{\lambda_z,\ord}.
\end{equation*}
\end{theo}

Let us remark on the hypothesis that $F^+ = \mbf Q$ in the second theorem. In general, if $z \in X_{\cl}(L)$ as in Theorem \ref{thm:we-did-it}, but $\rho_{z,v}$ is indecomposable at some place above $v \in \Sigma_p$ (e.g. $z$ satisfies $(\dagger)$), we can construct some interesting piece of $\Pi(\rho_{z,p})^{\ord}$ inside the completed cohomology. This piece is sufficient when $F^+ = \mbf Q$ but far from all of $\Pi(\rho_{z,p})^{\ord}$ in general.

In order to prepare the proof we will give some intermediate results on the principal series arising from points $z$ satisfying $(\dagger)$. For the reader interested in just the final steps of the proof of Theorem \ref{thm:we-did-it} see page \pageref{proof:big-theorem}, and page \pageref{proof:thm-special-case} for Theorem \ref{theo:F^+=Q}.

For the remainder of this section we work with points $z \in X_{\cl}(L)$ of weight-type $\tau$ satisfying $(\dagger)$. We recall from Proposition \ref{prop:good-points} that we have constructed a point $s_{\alpha_\tau}\cdot z$ which is either $z$ (if $z$ is non-critical) or the unique companion point associated to $z$ (if $z$ is critical). Explicitly, if $z = z_{\nc}^{(\sigma)}$ then $s_{\alpha_\tau}\cdot z = \left(\xi_{s_{\alpha_\tau}\cdot z}^{(13)}\delta_{B_{\Sigma_p}},\lambda_z\right)$ where\begin{equation}\label{eqn:xi}
\xi_{s_{\alpha_\tau}\cdot z,v} = \psi_{\sigma_v(3),v}z^{h_{\sigma_v(3),v}-h_{\tau_v(3),v}}\varepsilon^2 \otimes \psi_{\sigma_v(2),v}z^{h_{\sigma_v(2),v}-h_{\tau_v(2),v}}\varepsilon \otimes \psi_{\sigma_v(1)}z^{h_{\sigma_v(1),v}-h_{\tau_v(1),v}}.
\end{equation}
We first use results from the previous section to make explicit the locally analytic induction $\Ind_{B_{\Sigma_p}}^{G_{\Sigma_p}}(\xi_{s_{\alpha_\tau}\cdot z})^{\an}$. 

If $\tau \in S_3$ then let $\tau^\perp = (13)\tau(13)$. We note the following facts. Let $z = z_{\nc}^{(\sigma)}$ satisfy $(\dagger)$. Then:
\begin{itemize}
\item $s_{\alpha_\tau}\cdot z$ is $\alpha_{\tau^\perp}$-dominant.
\item If $z$ is critical then $\tau_v \sigma_v = (1)$ or $\tau_v\sigma_v = \tau_v^{\perp}$ (see Lemma \ref{lemm:simple-trans-non-critical}).
\end{itemize}
A short calculation from the definition \eqref{eqn:cont-definition}, and the previous bullet point, shows that if $z$ is critical and $\tau_v \neq \sigma_v$ then $\xi_{s_{\alpha_{\tau^\perp}}s_{\alpha_\tau}\cdot z, v}$ is unitary and
\begin{equation}\label{eqn:short-calc}
I(\rho_{z,v},\sigma_v) = \Ind_{B(\mbf Q_p)}^{\GL_3(\mbf Q_p)}(\xi_{s_{\alpha_{\tau^\perp}}s_{\alpha_\tau}\cdot z, v})^{C^0} = \Ind_{B(\mbf Q_p)}^{\GL_3(\mbf Q_p)}(\xi_{s_{\alpha_{\tau^\perp}}s_{\alpha_\tau}\cdot z, v})^{\an, \wedge}.
\end{equation}

\begin{prop}\label{prop:relations-each-place}
Suppose that $z=z_{\nc}^{(\sigma)}$ is of weight-type $\tau$ and satisfies $(\dagger)$. Then at each place $v$ we have
\begin{itemize}
\item If $\sigma_v = \tau_v$ then $\xi_{s_{\alpha_\tau}\cdot z,v}$ is unitary and $\Ind_{B(\mbf Q_p)}^{\GL_3(\mbf Q_p)}(\xi_{s_{\alpha_\tau}\cdot z,v})^{\an,\wedge} = I(\rho_{z,v},\sigma_v)$
\item If $\sigma_v \neq \tau_v$ then there is a short exact sequence
\begin{equation*}
0 \rightarrow \Ind_{B(\mbf Q_p)}^{\GL_3(\mbf Q_p)}(\xi_{s_{\alpha_\tau}\cdot z_{\nc}^{(\tau)},v})^{\an,\wedge} \rightarrow \Ind_{B(\mbf Q_p)}^{\GL_3(\mbf Q_p)}(\xi_{s_{\alpha_\tau}\cdot z,v})^{\an,\wedge} \rightarrow\Ind_{B(\mbf Q_p)}^{\GL_3(\mbf Q_p)}\left(\xi_{s_{\alpha_{\tau^\perp}}s_{\alpha_\tau}\cdot z,v}\right)^{\an,\wedge}\rightarrow 0
\end{equation*}

\end{itemize}
\end{prop}
\begin{proof}
The first statement follows directly from the description \eqref{eqn:xi}. We consider the second. By Proposition \ref{prop:identification} combined with \eqref{eqn:e2-e3-ext} and  \eqref{eqn:defn23} (or \eqref{eqn:e1-e2-ext} and \eqref{eqn:defn12}) we have a short exact sequence
\begin{equation*}
 0 \rightarrow I(\rho_{z,v},\tau_v) \rightarrow  \Ind_{B(\mbf Q_p)}^{\GL_3(\mbf Q_p)}(\xi_{s_{\alpha_\tau}\cdot z,v})^{\an,\wedge}  \rightarrow  I(\rho_{z,v},\sigma_v)  \rightarrow   0.
 \end{equation*} 
  The sequence of the proposition now follows from the first bullet point (for the sub) and \eqref{eqn:short-calc} for the quotient.
\end{proof}
The previous result, along with Lemma \ref{lemm:closed-under-bruhat}, allows us to compare the representations $\Ind_{B_{\Sigma_p}}^{G_{\Sigma_p}}(\xi_{s_{\alpha_\tau}\cdot z_{\nc}^{(\sigma)}})^{\an,\wedge}$ for various $\sigma$ (of weight-type $\tau$).

\begin{coro}\label{cor:bruhat-relations-inductions}
Suppose that $z = z_{\nc}^{(\sigma)}$ satisfies $(\dagger)$, is of weight-type $\tau$ and $\sigma' \in (S_3)_{\Sigma_p}$ such that $\tau \leq \sigma' < \sigma$ and $\ell(\sigma) = \ell(\sigma')+1$ in the Bruhat order on $(S_3)_{\Sigma_p}$. Then there is a short exact sequence
\begin{equation*}
0 \rightarrow \Ind_{B_{\Sigma_p}}^{G_{\Sigma_p}}(\xi_{s_{\alpha_\tau}\cdot z_{\nc}^{(\sigma')}})^{\an,\wedge} \rightarrow \Ind_{B_{\Sigma_p}}^{G_{\Sigma_p}}(\xi_{s_{\alpha_\tau}\cdot z_{\nc}^{(\sigma)}})^{\an,\wedge} \rightarrow \Ind_{B_{\Sigma_p}}^{G_{\Sigma_p}}(\xi_{s_{\alpha_{\sigma'^{-1}\sigma}}s_{\alpha_\tau}\cdot z_{\nc}^{(\sigma)}})^{\an,\wedge}  \rightarrow 0.
\end{equation*}
\end{coro}
\begin{proof}
Recall that Lemma \ref{lemm:closed-under-bruhat} says that if $\sigma'$ is between $\tau$ and $\sigma$ in the Bruhat order in $(S_3)_{\Sigma_p}$ then $z_{\nc}^{(\sigma')}$ satisfies $(\dagger)$, is of weight type $\tau$ and, moreover, if $\sigma'$ immediately precedes $\sigma$ then $\sigma'(1_{\Sigma_p-\{v_1\}}, \tau_{v_1}\sigma_{v_1}) = \sigma$ for a unique place $v_1 \in \Sigma_p$.

At the unique place $v_1$ for which $\sigma'_{v_1} \neq \sigma_{v_1}$, we have $\sigma'_{v_1} = \tau_{v_1}$. Thus by Proposition \ref{prop:relations-each-place} we have a short exact sequence of unitary $G(F_{v_1}^+)\simeq \GL_3(\mbf Q_p)$-representations
\begin{equation*}
0 \rightarrow \Ind_{B_{v_1}}^{G_{v_1}}(\xi_{s_{\alpha_\tau}\cdot z_{\nc}^{(\sigma')},v_1})^{\an,\wedge} \rightarrow 
\Ind_{B_{v_1}}^{G_{v_1}}(\xi_{s_{\alpha_\tau}\cdot z_{\nc}^{(\sigma)},v_1})^{\an,\wedge} \rightarrow \Ind_{B_{v_1}}^{G_{v_1}}(\xi_{s_{\alpha_{\sigma'^{-1}\sigma}}s_{\alpha_\tau}\cdot z_{\nc}^{(\sigma')},v_1})^{\an,\wedge} \rightarrow 0.
\end{equation*}
Since $\sigma'_{v} = \sigma_{v}$ for $v \neq v_1$, the corollary now follows by taking the completed tensor product of this sequence with $\widehat{\otimes}_{v \neq v_1} \Ind(\xi_{s_{\alpha_\tau}\cdot z_{\nc}^{(\sigma)},v})^{\an,\wedge}$.
\end{proof}

Returning to the relationship between locally analytic principal series and completed cohomology, recall that if $z$ satisfies $(\dagger)$ then Corollary \ref{coro:dagger-map} implies that there is a non-zero map
\begin{equation*}
I_{\an}(z)^{\wedge}  \rightarrow \widehat H^0(K^p)^{\lambda}_{L}
\end{equation*}
where $I_{\an}(z)^{\wedge} := \Ind(\xi_{s_{\alpha_\tau}\cdot z)})^{\an,\wedge}$ and $\tau$ is the weight-type of $z$. By Corollary \ref{cor:bruhat-relations-inductions}, and induction on $\ell(\sigma)$, the representation $I_{\an}(z)^{\wedge}$ contains a unique irreducible continuous principal series $I_{\an}(z_{\nc}^{(\tau)})^{\wedge}\simeq I_{C^0}(z_{\nc}^{(\tau)})$.

\begin{prop}\label{prop:multiplicity}
Suppose that  $z$ satisfies $(\dagger)$. Then:
\begin{enumerate}
\item The induced map 
\begin{equation}\label{eqn:this-one-is-an-iso}
\Hom_{G_{\Sigma_p}}\left(I_{\an}(z)^{\wedge}, \widehat H^0(K^p)^{\lambda_z}_{L}\right) \rightarrow \Hom_{G_{\Sigma_p}}\left(I_{C^0}(z_{\nc}^{(\tau)}), \widehat H^0(K^p)^{\lambda_z}_{L}\right)
\end{equation}
is injective.
\item Every non-zero map $I_{\an}(z)^{\wedge} \rightarrow \widehat H^0(K^p)^{\lambda_z}_{L}$ is injective.
\item If $z$ is non-critical (i.e. $\tau = (1)_{\Sigma_p}$) then the map \eqref{eqn:this-one-is-an-iso} is an isomorphism.
\end{enumerate}
\end{prop}
\begin{proof}
The first and third assertions are clear for the point $z = z_{\nc}^{(\tau)}$. The second assertion for $z = z_{\nc}^{(\tau)}$ follows from the irreducibility of $I_{C^0}(z_{\nc}^{(\tau)})$.

By Proposition \ref{prop:good-points}, $s_{\alpha_\tau}\cdot z$ is not a bad point. Note that $s_{\alpha_\tau}\cdot z$ is $\alpha_{\sigma'^{-1}\sigma}$-dominant for all $\tau \leq \sigma' < \sigma$.  Since $s_{\alpha_\tau}\cdot z$ is not bad, we have that
\begin{equation}\label{eqn:bad-hom}
\Hom_{G_{\Sigma_p}}\left(\Ind_{B_{\Sigma_p}}^{G_{\Sigma_p}}(\xi_{s_{\alpha_{\sigma'^{-1}\sigma}}s_{\alpha_\tau}\cdot z})^{\an,\wedge}, \widehat H^0(K^p)^{\lambda_z}_{L} \right) = (0)
\end{equation}
for any $\sigma'$ immediately preceding $\sigma$.

The first part of the proposition now follows by induction on $\ell(\sigma)$ from the case of $\sigma = \tau$, using Corollary  \ref{cor:bruhat-relations-inductions}.

For the second point, note that if $\tau \leq \sigma'$ and $\sigma'$ immediately precedes $\sigma$ in the Bruhat order then Proposition \ref{prop:relations-each-place} and Corollary \ref{cor:bruhat-relations-inductions} imply that $\Ind_{B_{\Sigma_p}}^{G_{\Sigma_p}}(\xi_{s_{\alpha_{\sigma'^{-1}\sigma}}s_{\alpha_\tau}\cdot z^{(\sigma')}_{\nc}})^{\an,\wedge}$ has a composition series whose irreducible factors are of the form 
\begin{equation*}
I_{C^0}(s_\beta s_{\alpha_\tau}\cdot z_{\nc}^{(\sigma'')})
\end{equation*}
 where $\tau \leq \sigma'' < \sigma$ and $\beta$ is some simple positive root. Since the points $s_{\alpha_\tau} \cdot z_{\nc}^{(\sigma'')}$ are all not bad, none of these irreducible constituents appear in the space $\widehat H^0(K^p)^{\lambda_z}_{L}$. Thus, not only does the space \eqref{eqn:bad-hom} vanish, but a non-zero map $I_{\an}(z^{(\sigma)}_{\nc})^{\wedge} \rightarrow \widehat H^0(K^p)_{L}^{\lambda_z}$ is injective if and only if its restriction to the subrepresentation $I_{\an}(z^{(\sigma')}_{\nc})^{\wedge}$ (for any/some $\tau \leq \sigma'$ immediately preceeding $\sigma$ in the Bruhat order) is injective. By induction on $\ell(\sigma)$ we conclude that a non-zero map $I_{\an}(z_{\nc}^{(\sigma)})^{\wedge} \rightarrow \widehat H^0(K^p)^{\lambda_z}_{L}$ is injective if and only if its restriction to $I_{C^0}(z_{\nc}^{(\sigma)})$ is injective. Thus the second statement is true, by the first paragraph of the proof.

For the third point, as $z = s_{\alpha_\tau}\cdot z$ is non-critical, hence not bad by Proposition \ref{prop:bad-implies-critical}, we have by Theorem \ref{thm:chenevier-classicality}
\begin{equation*}
\dim F(\mc C_{\theta_z\delta_k},K^p)^{(\lambda_z),\mc A_p} = \dim  F(W_k^\vee,K^p)^{(\lambda_z),\mc A_p=\theta_z}.
\end{equation*}
Taking $z = z_{\nc}^{(1)}$ we also have
\begin{equation*}
\dim F(\mc C_{\theta_{z_{\nc}^{(1)}} \delta_k},K^p)^{(\lambda_{z_{\nc}^{(1)}}),\mc A_p}
= \dim  F(W_k^\vee,K^p)^{(\lambda_{z_{\nc}^{(1)}}),\mc A_p=\theta_{z_{\nc}^{(1)}}}.
\end{equation*}
But by the generic ordinary assumption on $z$, the two {\em classical} spaces on the right hand side of the previous two equations have the same dimension. After taking eigenspaces of all the spaces involved and passing to the Jacquet module side (via Proposition \ref{prop:comparision}), Theorem \ref{theo:analytic-maps} implies the final point of our proposition.
\end{proof}

We are now in a position to prove Theorem \ref{thm:we-did-it}. 

\begin{proof}[Proof of Theorem \ref{thm:we-did-it}]\label{proof:big-theorem}
Suppose that $z \in X_{\cl}(L)$ is generic ordinary and that $\rho_{z,v}$ is totally indecomposable above $p$. Without loss of generality we assume that $z = z_{\nc}^{(1)}$. We then consider the collection of points
\begin{equation*}
\mscr Z = \{ z_{\nc}^{(\sigma)} : \sigma \text{ is simple}\} = \{z_{\nc}^{(\sigma)} : z_{\nc}^{(\sigma)} \text{ satisfies $(\dagger)$}\},
\end{equation*}
which we partially order by the element $\sigma \in (S_3)_{\Sigma_p}$. We then let
\begin{align*}
\mscr Y &= \{ z_{\nc}^{(\sigma)} \in \mscr Z : \sigma \text{ is maximal in the Bruhat order on $\mscr Z$}\}\\
&= \{z_{\nc}^{(\sigma)} \in \mscr Z : \sigma_v \neq (1) \text{ for any $v \in \Sigma_p$}\}.
\end{align*}
Let $\Pi(\rho_{z,p})_{\mscr Y}$ be the amalgamated product over $I_{C^0}(z_{\nc}^{(1)})$
\begin{equation*}
\Pi(\rho_{z,p})_{\mscr Y} := \prod_{\substack{I_{C^0}(z_{\nc}^{(1)}) \\ z \in \mscr Y}} I_{\an}(z)^{\wedge}.
\end{equation*}
By using Proposition \ref{prop:relations-each-place}, we have $\Pi(\rho_{z,p})_{\mscr Y} = \Pi(\rho_{z,p})^{\ord}$. Thus, Proposition \ref{prop:multiplicity}(3) implies that the natural inclusion $I_{C^0}(z_{\nc}^{(1)}) \subset \Pi(\rho_{z,p})^{\ord}$ induces an isomorphism
\begin{equation*}
\Hom_{G_{\Sigma_p}}(\Pi(\rho_{z,p})^{\ord}, \widehat H^0(K^p)^{\lambda_z}_{L}) \simeq \Hom_{G_{\Sigma_p}}(I_{C^0}(z_{\nc}^{(1)}), \widehat H^0(K^p)^{\lambda_z}_{L}).
\end{equation*}
Since the left hand space is non-zero, and $I_{C^0}(z_{\nc}^{(1)})$ is irreducible, Proposition \ref{prop:multiplicity}(2) applied to $z$ implies there is an embedding
\begin{equation*}
\Pi(\rho_{z,p})^{\ord} \hookrightarrow \widehat H^0(K^p)_{L}^{\lambda_z}.
\end{equation*}
%
This completes the proof of Theorem \ref{thm:we-did-it}.
\end{proof}
\begin{rema}
Before moving on to prove Theorem \ref{theo:F^+=Q} we make a small remark. In the proof we only used the points $z_{\nc}^{(\sigma)}$ which appear in the set $\mscr Y$. But equivalently (at least after taking $\Hom$) the amalgamated sum could have been taken over the whole set $\mscr Z$ by Proposition \ref{cor:bruhat-relations-inductions} and Proposition \ref{prop:multiplicity}. Thus we have used all the points that satisfied $(\dagger)$ to produce the representation $\Pi(\rho_{z,p})^{\ord}$ inside $\widehat H^0(K^p)^{\lambda_z}_{L}$.
\end{rema}

We now give the proof of Theorem \ref{theo:F^+=Q}. It is in the same spirit as the proof of Theorem \ref{thm:we-did-it} though we use the existence of companion points.
\begin{proof}[Proof of Theorem \ref{theo:F^+=Q}]\label{proof:thm-special-case}
We fix the point $z = z_{\nc}^{(1)} \in X_{\cl}(L)$. Our working assumption is that $\rho_{z,p}$ is indecomposable, but not totally indecomposable. Thus going back to Lemma \ref{lemm:simple-trans-non-critical} among the six classical points $z_{\nc}^{(\sigma)}$ there are two which are critical and four which are not.

Specifically, we denote by $\tau$ the unique transposition such that $z_{\nc}^{(\tau)}$ is critical (of weight-type $\tau$). In this notation, the four points $z_{\nc}^{(1)}, z_{\nc}^{(\tau^{\perp})}$, $z_{\nc}^{(\tau)}$ and $z_{\nc}^{(\tau\tau^{\perp})}$ all satisfy $(\dagger)$; the first two in that list are non-critical and the second two are critical of weight-type $\tau$. By comparing Definition \ref{defi:ord} and \eqref{eqn:xi} via Proposition \ref{prop:identification} we get that 
\begin{equation*}
\Pi(\rho_{z,p})^{\ord} = I_{\an}(z_{\nc}^{(\tau^{\perp})})^{\wedge} \oplus I_{\an}(z_{\nc}^{(\tau\tau^{\perp})})^{\wedge} =: \Pi_1 \oplus \Pi_2.
\end{equation*} 
Moreover, the Jordan-Holder constituents of the two summands appearing on the right-hand side are pairwise distinct (see Corollary \ref{cor:bruhat-relations-inductions} for example, and compare with the proof of Theorem \ref{thm:we-did-it}). Combining this observation with Proposition \ref{prop:multiplicity}(2) we see that a non-zero morphism $\Pi(\rho_{z,p})^{\ord} \rightarrow \widehat H^0(K^p)^{\lambda_z}_{L}$ is injective. Since $\Hom_{G_{\Sigma_p}}(\Pi_i,\widehat H^0(K^p)^{\lambda_z}_{L}) \neq 0$, this completes the proof of Theorem \ref{theo:F^+=Q}.
\end{proof}


\begin{thebibliography}{9}


\bibitem[BC]{bc} J. Bella\"iche, G. Chenevier, ``Families of Galois representations and Selmer groups", Asterisque 324, Soc. Math. France, 314 p. (2009).

\bibitem[Ber]{ber2} J. Bergdall, ``Ordinary modular forms and companion points on the eigencurve", J. Number Theory 134 (2014), no. 1, 226-239.

\bibitem[Ber2]{ber-thesis} J. Bergdall, ``On the variation of $(\varphi,\Gamma)$-modules over $p$-adic families of automorphic forms'', PhD thesis, Brandeis University.

\bibitem[Bre]{br} C. Breuil, ``Vers le socle localement analytique pour $\GL_n$, II", preprint (2013).


\bibitem[BH]{bh} C. Breuil, F. Herzig, ``Ordinary representations of $G(\mbf{Q}_p)$ and fundamental algebraic representations", preprint (2013).

\bibitem[BE]{be} C. Breuil, M. Emerton, ``Representations p-adiques ordinaires de $\GL_2(\mbf{Q}_p)$ et compatibilite local-global", Asterisque 331 (2010) 255-315.

\bibitem[Che1]{che} G. Chenevier ``On the infinite fern of Galois representations of unitary type", Ann. Sci. E.N.S. 44, 963-1019 (2011).

\bibitem[Che2]{che2} G. Chenevier ``Une application des vari\'et\'es de {H}ecke des groups unitaires'', preprint. (2009)


\bibitem[Cho]{cho3} P. Chojecki, ``Density of crystalline points on unitary Shimura varieties", Int. J. Number Theory, Vol. 9, No. 3 (2013) 1-15.

\bibitem[Col]{col} P. Colmez ``R\'epresentations triangulines de dimension 2", Asterisque 319 (2008), 213-258.

\bibitem[CDP]{cdp} P. Colmez, G. Dospinescu, V. Paskunas, ``The $p$-adic local Langlands correspondence for $\GL_2(\mbf{Q}_p)$", preprint (2013).

\bibitem[Eme1]{em1} M. Emerton, ``On the interpolation of systems of eigenvalues attached to automorphic Hecke eigenforms", Invent. Math. 164 (2006), no. 1, 1-84.

\bibitem[Eme2]{em2} M. Emerton, ``Jacquet modules of locally analytic representations of p-adic reductive groups I. Construction and first properties", Ann. Sci. E.N.S. 39 (2006), no. 5, 775-839.

\bibitem[Eme3]{em3} M. Emerton, ``Jacquet modules of locally analytic representations of p-adic reductive groups II. The relation to parabolic induction", to appear in J. Institut Math. Jussieu. 

\bibitem[Eme4]{em4} M. Emerton, ``A local-global compatibility conjecture in the p-adic Langlands programme for $\GL_{2 / \mathbb{Q}}$", Pure and Applied Math. Quarterly 2 (2006), no. 2, 279-393. 

\bibitem[Eme5]{em5} M. Emerton, ``Local-global compatibility in the p-adic Langlands programme for $\GL_2/\mb{Q}$", preprint (2011). 


\bibitem[Jon]{jones} O. Jones, ``An analogue of the BGG resolution for locally analytic principal series", J. Number Theory 131 (2011) 1616--1640.

\bibitem[Hum]{hum} J. Humphreys, ``Representations of semisimple Lie algebras in the BGG category $\mc O$'',
Graduate Studies in Mathematics, 94. American Mathematical Society, Providence, RI,  2008. xvi+289 pp.

\bibitem[KPX]{kpx} K. Kedlaya, J. Pottharst and L. Xiao, ``Cohomology of arithmetic families of $(\phi, \Gamma)$-modules", to appear in J.A.M.S. 

\bibitem[Kis]{kis-omf} M. Kisin, ``Overconvergent modular forms and the Fontaine-Mazur conjecture", Invent. Math. 153(2) (2003), 373-454. 

\bibitem[Loe]{loe} D. Loeffler, ``Overconvergent algebraic automorphic forms", Proc. London Math. Soc. 102 (2011), no. 2, 193-228.

\bibitem[Maz]{maz} B. Mazur ``Deforming Galois representations", in ``Galois groups over $\mbf{Q}$", ed. Ihara et al., Springer-Verlag (1987).

\bibitem[Liu]{liu} R. Liu, ``Semi-stable periods of finite slope families", preprint (2013).

\bibitem[Rou]{rou} R. Rouquier, ``Caract\'erisation des caract\`eres et pseudo-caract\`eres", J. Algebra, 180(2), 571-586 (1996).

\bibitem[ST1]{st} P. Schneider, J. Teiltelbaum, ``Locally analytic distributions and p-adic representation theory, with applications to $\GL_2$", J.A.M.S. 15, 443-468 (2002). 

\bibitem[ST2]{st4} P. Schneider, J. Teiltelbaum, ``Duality for admissible locally analytic representations", Represent. Theory 9, 297-326 (electronic), 2005.  

\bibitem[Tay]{tay} R. Taylor, ``Galois representations associated to {S}iegel modular forms of low weight", Duke Math. J. 63, 281-332 (1991).


\end{thebibliography}
\end{document}